\documentclass{article}

\usepackage{amsmath,amsthm,mathrsfs}
\usepackage{bbm}
\usepackage{amsfonts}
\usepackage{anysize}
\usepackage{titlesec}
\usepackage{hyperref}
\usepackage[nottoc]{tocbibind}
\usepackage[titletoc]{appendix}
\usepackage[font=small,labelfont=bf]{caption}
\usepackage{tikz}
\usetikzlibrary{decorations.pathmorphing}
\usepackage{enumerate}

\usepackage[comma, sort&compress, authoryear]{natbib} %For bibliography

\newtheoremstyle{mystyle}%                % Name
  {}%                                     % Space above
  {}%                                     % Space below
  {}%                                     % Body font
  {}%                                     % Indent amount
  {}%                                     % Theorem head font
  {.}%                                    % Punctuation after theorem head
  { }%                                    % Space after theorem head, ' ', or \newline
  {{\underline{\thmname{#1}\thmnumber{ #2}~\thmnote{(#3)}}}}%                                     % Theorem head spec (can be left empty, meaning `normal')

\newtheoremstyle{mystyle2}%                % Name
  {}%                                     % Space above
  {}%                                     % Space below
  {\itshape}%                                     % Body font
  {}%                                     % Indent amount
  {}%                                     % Theorem head font
  {.}%                                    % Punctuation after theorem head
  { }%                                    % Space after theorem head, ' ', or \newline
  {{\underline{\thmname{#1}\thmnumber{ #2}~\thmnote{(#3)}}}}%                                     % Theorem head spec (can be left empty, meaning `normal')

\theoremstyle{definition}
\theoremstyle{mystyle}
\newtheorem{defi}{Definition}[section]
\newtheorem{rmk}[defi]{Remark}
\newtheorem{lemma}[defi]{Lemma}

\newtheorem*{lemma*}{Lemma}

\theoremstyle{plain}
\theoremstyle{mystyle2}
\newtheorem{thm}[defi]{Theorem}
\newtheorem{prop}[defi]{Proposition}
\newtheorem{cor}[defi]{Corollary}

\makeatletter
\newcommand{\R}{\mathbb{R}}
\newcommand{\N}{\mathbb{N}}

\newcommand{\Z}{\mathbb{Z}}
\newcommand{\p}{\mathbb{P}}
\newcommand{\E}{\mathbb{E}}
\newcommand{\var}{\mathrm{var}}

\newcommand{\e}{\mathrm{e}}
\newcommand{\dd}{\mathrm{d}}

\newcommand{\1}{\mathbbm{1}}

\newcommand{\supp}{\mathrm{supp}}

\newcommand{\Cov}{\mathrm{Cov}}
\newcommand{\sk}{\vspace*{\baselineskip}}

\newcommand{\grad}{\nabla}

\title{Hydrodynamic limit of the Symmetric Exclusion Process on a compact Riemannian manifold}
\author{Bart van Ginkel\footnote{G.J.vanGinkel@tudelft.nl} \and Frank Redig\footnote{F.H.J.Redig@tudelft.nl}}
\date{%
    TU Delft\\
    \today
}
\begin{document}
\maketitle

\begin{abstract}
We consider the symmetric exclusion process on suitable random grids that approximate a compact Riemannian manifold.
We prove that a class of random walks on these random grids converge to Brownian motion on the manifold.
We then consider the empirical density field of the symmetric exclusion process and prove that it converges to the solution of the heat equation on the manifold.
\end{abstract}%\newpage
\section{Introduction}

Hydrodynamic limits of interacting particle systems is a well established subject. A large variety of parabolic equations (such as the non-linear heat equation) and hyperbolic conservation laws have been obtained from microscopic stochastic particle systems; see \cite{kipnis1999scaling, demasi2006mathematical, seppalainen2008translation} for overviews.
Usually, the setting here is that in the underlying particle system the particles move on the lattice $\Z^d$, and after rescaling the limiting
partial differential equation is defined on $\R^d$, or on a subdomain of $\R^d$ such as an interval, where then equations with boundary conditions on
the ends of the interval are derived (e.g. Dirichlet boundary conditions for the case where at the right and left end the system is coupled to a reservoir fixing the density of particles, see \cite{goncalves}).

Motivated e.g.\ by the study of the motion of proteins in a cell-membrane, or more general motion of particles on curved interfaces,
it is clear that there are many relevant physical systems of which the macroscopic motion
takes place on a Riemannian manifold rather than on Euclidean space. It is the aim of this paper to provide first steps in this direction, by considering the simplest interacting particle system on a suitable discretization of a Riemannian manifold and proving its hydrodynamic limit.
The symmetric exclusion process is a well-known and well-studied interacting particle system 
for which in standard setting it is rather straightforward to obtain the hydrodynamic limit using duality.
Duality allows to translate the one-particle scaling limit, i.e., the fact that the rescaled single particle
position converges to Brownian motion to the fact that the hydrodynamic limit of the particle system is the
diffusion equation. Another manifestation of duality is the fact that the microscopic equation for the expectation of the density field
is already a closed equation. We consider the symmetric exclusion process on a suitable discretization (a notion defined more precisely below)
of a compact Riemannian manifold and prove that its empirical density field, after appropriate rescaling, converges to the solution of the heat equation
on the manifold.
To obtain this result, we start in section 2 by studying the invariance principle of a class of geodesic random walks, thereby extending earlier results of
\cite{jorgensen1975central}. These random walks are shown to converge to Brownian motion, via the technique of generator convergence. Next, in section 3, we define a notion of
``uniformly approximating grids'' and show that choosing uniformly $N$ points on the manifold, and connecting them via a kernel depending on the Riemannian distance yields a weighted graph such that the corresponding random walk converges (as the number of random points tends to infinity) to a geodesic random walk which in turn scales to Brownian motion. We also formulate abstract conditions on approximating grids ensuring the convergence of the weighted random walk to Brownian motion. In particular, convergence of the empirical distribution to the normalized Riemannian volume in Kantorovic distance is shown to be sufficient, i.e. we show that in that setting weights can be chosen such that the corresponding random walk converges to Brownian motion. We give several examples of such suitable grids.
Finally, in section 4, we define the exclusion process on such suitable grids (defined in section 3) and show that its empirical density converges to the solution of the heat equation, following the proof from \cite{seppalainen2008translation}.

%Main idea: construction of grids and random walks on them such that the discrete Laplacian converges to the Laplace-Beltrami operator. Then use this to %prove a hydrodynamic limit result.\\
%\begin{itemize}
%    \item First some preliminary results on geodesic random walk and invariance principle
%    \item Then construction of random walk on grids convergence in Kantorovich sense, proof that we get convergence of the Laplacian and result on existence of such grids.
%    \item Finally application: hydrodynamic limit result.
%\end{itemize}%\newpage

\section{The invariance principle for a class of geodesic random walks}\label{part3}

Let $M$ be an $n$-dimensional, compact and connected Riemannian manifold. Then we know that $M$ is complete and hence geodesically complete. The main purpose of this section is to define the geodesic random walk and to show that it approximates Brownian motion when appropriately rescaled (in time and space). Such random walks and this so-called invariance principle have been studied before (\cite{jorgensen1975central} and in a special case \cite{blum1984note}). However we will directly obtain results that are tailor-made to apply them in section~\ref{part2}. In particular, we will obtain general assumptions on the jumping distributions of the geodesic random walk for it to converge to Brownian motion. %Moreover, these results are in some cases more general than the results in~\cite{jorgensen1975central} since we will not have to assume that the jumps are identically distributed.\\
In section~\ref{1generators}, we define the geodesic random walk and show convergence of the generators to the generator of Brownian motion under certain assumptions on the jumping distributions.
%In section~\ref{2dirforms} give the analogous result for Dirichlet forms. All of this is done under general assumptions on the stepping distribution of the random walk.
Section~\ref{3stepdist} is devoted to finding out which distributions satisfy these assumptions.
%In section~\ref{1generators}, we will first define a geodesic random walk on $M$. Then we show that the generators of these random walks converge to the generator of Brownian motion when the step size decreases and the step rate increases. This implies that the corresponding processes converge in distribution in the path space. In section~\ref{2dirforms} we first discuss the existence and expression of the corresponding Dirichlet forms. Then we show convergence (in a similar way to the convergence of the generators) and we make a remark on what the convergence of Dirichlet forms says about the convergence of processes. All of this is done under general assumptions on the stepping distribution. Section~\ref{3stepdist} is devoted to finding out which distributions satisfy these assumptions.\\
\\
\subsection{Convergence of the generators}\label{1generators}
\textbf{The process}\\
Let $\{\mu_p,p\in M\}$ be a collection of positive, finite measures where each $\mu_p$ is a measure on $T_pM$. The measure $\mu_p$ represents the rate to jump in a particular direction of $T_pM$. More precisely, the Markov process $X^N=\{X^N_t,t\geq 0\}$ associated to $\{\mu_p,p\in M\}$ has generator
\begin{equation*}
    L_Nf(p)=\int_{T_pM} f(p(1/N,\eta))-f(p) \mu_p(\dd \eta),
\end{equation*}
where for a vector $\xi\in T_pM$ we denote the geodesic through $p$ with tangent vector $\xi$ at $p$ by $p(\cdot,\xi)$.
We denote the corresponding semigroup by
\begin{equation*}
    S^N_tf(p)=\E_pf(X^N_t).
\end{equation*}
Both of these have the continuous functions on the manifold $C(M)$ as their domain.\\
\\
We interpret this process as follows.
When the process $X^N$ is at a point $p$, it chooses a random direction $\eta$ from $T_pM$ with rates given by $\mu_p$ (i.e. it waits for an exponential time with rate $\mu_p(T_pM)$ and then independently picks a vector according to the probability distribution $\frac{\mu_p}{\mu_p(T_pM)}$). Then the process jumps to the position $p(1/N,\eta)$ that is reached by following the geodesic through $p$ in the direction of $\eta$ for time $\frac{1}{N}$. This situation is sketched in figure~\ref{geodRW}. We assume that choosing random directions happens independently. In this section we will specify restrictions that the measures $\mu_p$ should satisfy. Later (in section~\ref{3stepdist}), we will show that we can take $\mu_p$ to be for instance the uniform distribution on the unit tangent vectors at $p$.\\
\\
\begin{figure}
\begin{tikzpicture}[scale=0.92]
\draw (4,4) circle [radius=4];
\draw (0,4) arc (180:360:4 and 0.7);

%arc one
\draw [fill] (2,6) circle [radius=1.5pt];
\node [below] at (2,6) {$p_0$};
\draw [dashed] (6.15,6.7) arc (80:130:5 and 3);
\draw [->] (2,6) to (4,7);
\node at (3,6.7) {$\eta_0$};

%arc two
\draw [fill] (6.15,6.7) circle [radius=1.5pt];
\node [above] at (6.15,6.7) {$p_1$};
\draw [dashed] (6.15,6.7) arc (80:65:6 and 40);
\draw [->] (6.15,6.7) to (7.2,5.24);
\node at (7.1,6) {$\eta_1$};

%arc three
\draw [fill] (7.7,3.43) circle [radius=1.5pt];
\node [below] at (7.7,3.43) {$p_2$};
\draw [dashed] (7.65,3.5) arc (30:60:5 and 4);
\draw [->] (7.65,3.5) to (6.8,4.5);
\node at (7,3.9) {$\eta_2$};

%arc four
\draw [fill] (5.75,5) circle [radius=1.5pt];
\node [above] at (5.75,5) {$p_3$};
\draw [dashed] (5.75,5) arc (80:70:4 and 40);
\draw [->] (5.75,5) to (6.3,4);
\node at (5.75,4.3) {$\eta_3$};

%point five
\draw [fill] (6.45,3.2) circle [radius=1.5pt];
\node [below] at (6.45,3.2) {$p_4$};

\node at (2.5,4.5) {$M=S^2$};

\end{tikzpicture}
\includegraphics[width=0.5\textwidth]{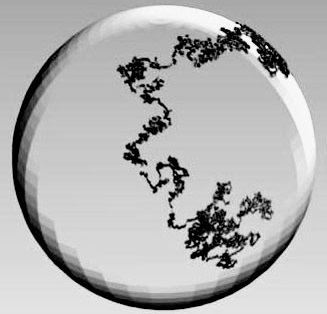}
\caption{Left: geodesic random walk on a sphere. Right: Brownian motion on a sphere (source: \nolinkurl{https://en.wikipedia.org/wiki/Brownian_motion}).}
\label{geodRW}
\end{figure}
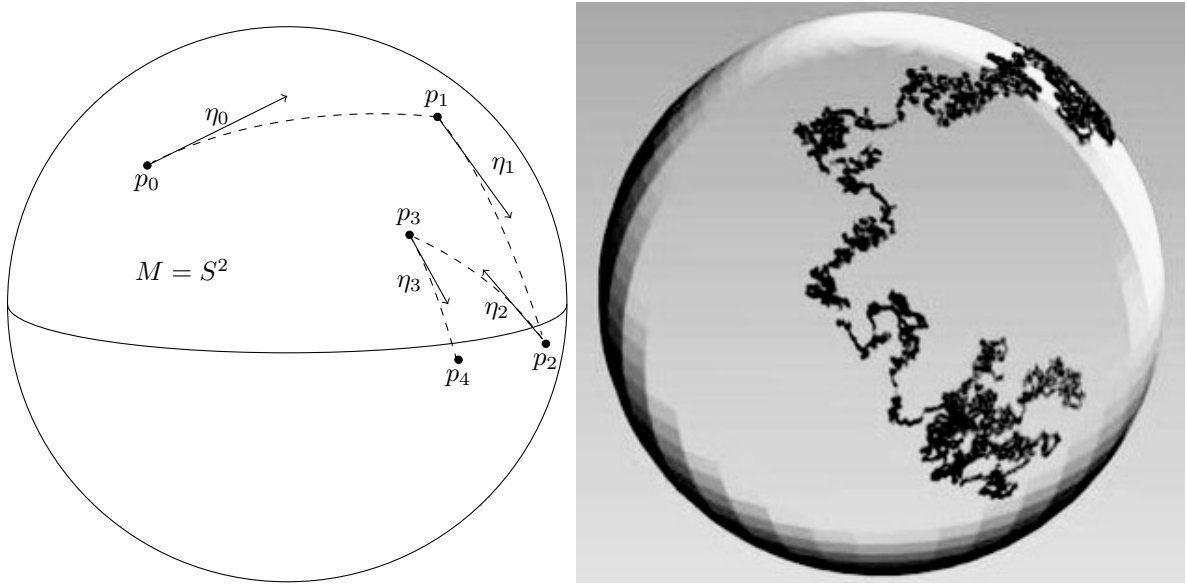

\textbf{The $\R^n$ case}\\
Before we go into the general case, we illustrate the above in $\R^n$.
%In the introduction we already described the result in $\R^n$, but we will repeat it here in a way that is a bit more technical to see that the process described above generalizes it. We have seen that i
In $\R^n$ the exponential map is simply addition if we identify $T_p\R^n$ with $\R^n$ itself. So in that case from a point $p$ the process moves to $p(1/N,\eta)=p+\frac{1}{N}\eta$ where $\eta$ is chosen from $T_p\R^n=\R^n$ randomly. This means that the discrete time jumping process when jumping as described above, can be denoted by $S^N_m=\sum_{i=1}^m\frac{1}{N}\eta_i=\frac{1}{N}\sum_{i=1}^m\eta_i$ where $\eta_j$ is drawn from $T_{S_{j-1}}\R^n=\R^n$ according to some distribution. Now let $\{N_t,t\geq 0\}$ be a Poisson process with rate one and define $X^N_t=S_{N_t}$. Then $X$ makes the same jumps as $S$, but after independent exponential times. We see that $X^N=\{X^N_t,t\geq 0\}$ satisfies the description above. Now the invariance principle tells us that under some conditions on the jumping rates $X^N_{tN^2}\rightarrow B_t$ in distribution as $N$ goes to infinity, where $B$ is Brownian motion. We show the analogous result in the more general setting of a manifold.\\
\\
\textbf{Aim}\\
We denote the Laplace-Beltrami operator on the manifold by $\Delta_M$. The rest of this section will be devoted to the proof of the following result.
\begin{prop}\label{genconv}
Suppose that in the situation above we have:
\begin{itemize}
    \item $\sup_{p\in M} \sup_{\eta\in\supp \mu_p} ||\eta||<\infty$
    \item $\sup_{p\in M} \mu_p(T_pM)<\infty$
    \item $\int \eta^i \mu_p(\dd \eta)=0$ and $\int\eta^i\eta^j\mu_p(\dd \eta)=g^{ij}(p)$ in each coordinate system around $p$
\end{itemize}
Then for $f\in C^\infty$: $N^2L_Nf\rightarrow \frac{1}{2} \Delta_Mf$ uniformly on $M$.
\end{prop}
The first assumption requires that the supports of the measures and their total masses are bounded uniformly over all points of the manifold. We will loosely say that the measures are uniformly compactly supported and uniformly finite.
%We would like to show that a rescaled version of the random walk described above converges to Brownian motion as $N$ goes to infinity (so that the left picture of figure~\ref{geodRW} converges to the right one). In $\R^n$ we had to scale time by $N^2$ if space is scaled by $N$. This turns out to be the right scaling in the general case as well. \\
Since $C^\infty(M)$ is a core for $\frac{1}{2}\Delta_M$ (\cite{strichartz1983analysis}), the Trotter-Kurtz theorem (see~\cite{kurtz1969extensions}) implies the following corollary.
\begin{cor}
In the situation of proposition~\ref{genconv} the geodesic random walk converges to Brownian motion in distribution in $D([0,\infty),M)$ (the space of cadlag maps $[0,\infty)\rightarrow M$).
\end{cor}
Note that if we denote the random variable corresponding to $\mu_p$ by $\zeta_p$, the second requirement of proposition~\ref{genconv} is that (in any coordinate system) $\E \zeta_p^i = 0$ and $\Cov(\zeta_p^i,\zeta_p^j)=g^{ij}(p)$. This shows that the mean vector $m$ of $\zeta_p$ satisfies $m=0$ and the covariance matrix $\Sigma$ satisfies $\Sigma=(g^{ij})(p)$. In $\R^n$, this simplifies to $\E \zeta_p^i = 0$ and $\Cov(\zeta_p^i,\zeta_p^j)=\delta^i_j$. This is satisfied for instance when $\mu_p$ is the uniform distribution on the sphere with radius $\sqrt N$ in $\R^n$. Section~\ref{3stepdist} deals with the question which measures satisfy the restrictions above. Some examples will be given at the end of that section as well.
\begin{rmk} Although we study the jumping distributions later, something that can already be seen now, is that we do not require any relation between jumping measures at different points of the manifold (apart from the uniform bounds on the support and the total mass). This means that our result does not require the jumping measures to be identically distributed, so it really generalizes~\cite{jorgensen1975central}.
\end{rmk}\sk

\textbf{Choosing suitable charts}\\
Let $f$ be a fixed smooth function from now on. Since we want the convergence $N^2L_N f\rightarrow \frac{1}{2} \Delta_Mf$ to be uniform on $M$, we cannot just consider this problem pointwise. To deal with this, we will choose specific coordinate charts.\\
Let $\rho$ denote the original metric of the manifold and let $d$ denote the metric that is induced by the Riemannian metric. Recall that these metrics induce the same topology. This means that we do not cause confusion when we speak about open and closed sets, continuous maps and compactness without explicitly mentioning the metric. For each $p\in M$, let $(x_p,U_p)$ be a coordinate chart for $M$ around $p$. $U_p$ is open with respect to $\rho$ and hence with respect to $d$. This means that there is some $\epsilon_p>0$ such that $G_p:=\overline{B_d(p,\epsilon_p)}\subset U_p$. Now define $O_p=B_d(p,\epsilon/2)$. Since $M$ is compact, we can find $p_1,..,p_m$ such that $M\subset \cup_i O_{p_i}$. We have the following easy statement.
\begin{lemma}
Let $(g_k)_{k=1}^\infty$ and $g$ be functions $M\rightarrow \R$. If $g_k\rightarrow g$ uniformly on each $O_{p_i}$, then $g_k\rightarrow g$ uniformly on $M$.
\end{lemma}
\begin{proof}
Let $\epsilon>0$. For each $i$ there is an $N_i\in\mathbb{N}$ such that for all $k\geq N_i: \sup_{O_{p_i}}|g_k(q)-g(q)|<\epsilon$. Set $N=\max_{1\leq i\leq m}N_i$ and let $q\in M$. Then there is a $j$ such that $q\in O_{p_j}$. Now for all $k\geq N$, we see $k\geq N_j$, so $|g_k(q)-g(q)|\leq \sup_{O_{p_i}}|g_k(s)-g(s)|<\epsilon$. This shows that $\sup_M |g_k(q)-g(q)|\leq\epsilon$. Hence $g_k\rightarrow g$ uniformly on $M$.
\end{proof}
Now let $j\in\{1,..,m\}$ be fixed. Call $O:=O_{p_j}$, $\epsilon:=\epsilon_{p_j}$, $x:=x_{p_j}$, $G:=G_{p_j}$ and $U:=U_{p_j}$ (this situation is shown in figure~\ref{proof_charts}). Because of the lemma, it suffices to show that $N^2L_N f\rightarrow \frac{1}{2} \Delta_Mf$ uniformly on $O$.\\
\\
\textbf{Technical considerations}\\
To obtain good estimations later, we will need that $p(s,\eta)$ is still in our coordinate system $(x,U)$ and even in the set $G$ when $|s|\leq \frac{1}{N}$ for $N$ large enough. Since the convergence must be uniform, how large $N$ must be can not depend on the point $p$. The following lemma tells us how to choose such $N$.

\begin{lemma}\label{bndlem}
Call $K=\sup_{p\in M} \sup_{\eta\in\supp \mu_p} ||\eta||<\infty$ (by assumption). Choose $N_\epsilon\in\N$ such that $\frac{1}{N_\epsilon}<\frac{\epsilon}{2K}$. Then for all $p\in O$ and $N\geq N_\epsilon$ we see
\begin{equation*}
    \forall |s|\leq \frac{1}{N}: p(s,\eta)\in G.
\end{equation*}
\end{lemma}
\begin{proof}
Let $N\geq N_\epsilon$ and let $p\in O$. The situation of the proof is visually represented in figure~\ref{proof_charts}. Fix $s\in(-\frac{1}{N},\frac{1}{N})$. Without loss of generality assume $s>0$. Note that the speed of the geodesic $p(\cdot,\eta)$ equals $||\eta||$, so at time $s$, it has traveled a distance $s||\eta||$ from $p$. This means that there is a path of length $s||\eta||$ from $p(s,\eta)$ to $p$, so $d(p(s,\eta),p)\leq s||\eta||\leq \frac{1}{N}K \leq \frac{1}{N_\epsilon}K<\epsilon/2$. Since $p\in O$, we know $d(p,p_j)<\epsilon/2$. Now the triangle inequality shows that $d(p_j,p(s,\eta))\leq d(p_j,p)+d(p,p(s,\eta))< \epsilon/2+\epsilon/2=\epsilon$. This implies that $p(s,\eta)\in B_d(p_j,\epsilon)\subset G$.
\end{proof}

%Note that to do this, we must assume that the support of $\mu_p$ is compact for each $p$ and in some sense uniformly in $p$ (so $K$ as above exists). If not, the geodesics can go arbitrarily fast, so they generally do not stay close enough to the starting point.\\
%\textit{\textbf{Assumption 1:} $\mu_p$ has compact support for each $p\in M$. Moreover %}
%\begin{equation*}
%    \sup_{p\in M}\sup_{\eta\in\supp\mu_p}||\eta||<\infty.
%\end{equation*}
%We know now that $K$ as in lemma~\ref{bndlem} exists. We f
Fix $N_\epsilon$ as in the lemma and take $N$ larger than $N_\epsilon$.\\
\\
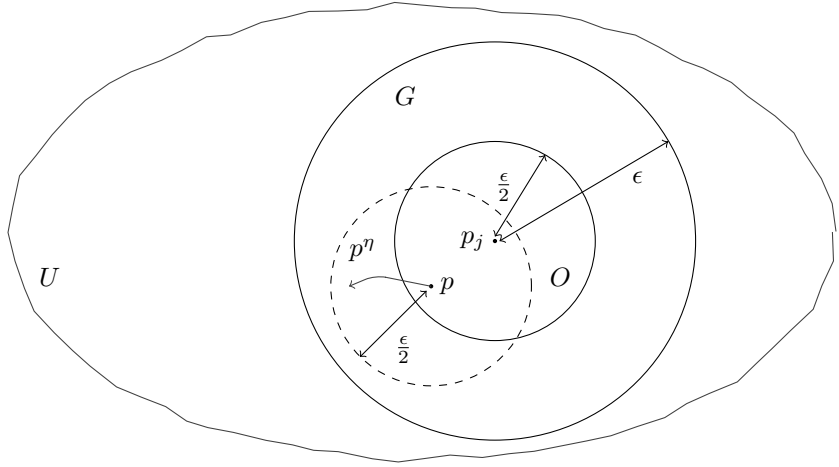
\begin{figure}[ht]
\begin{centering}
\begin{tikzpicture}[scale=1.2,pencildraw/.style={
    black!75,
    decorate,
    decoration={random steps,segment length=10pt,amplitude=1.5pt}
    }
]
%wobbly circle
\draw[pencildraw] (9.5,2.5) arc (0:-360:45mm and 25mm);
%name
\node [right] at (0.7,2) {$U$};

%big circle
\draw (5.8,2.4) circle [radius=2.2];
\node [left] at (5.8,2.4) {$p_j$};
\draw [fill] (5.8,2.4) circle [radius=0.5pt];
%name
\node [right] at (4.6,4) {$G$};
%arrow
\draw [<->] (5.85,2.4) -- (7.7,3.5);
\node [right] at (7.2,3.1) {$\epsilon$};

%small circle
\draw (5.8,2.4) circle [radius=1.1];
%name
\node [right] at (6.3,2) {$O$};
%arrow
\draw [<->] (5.8,2.45) -- (6.35,3.35);
\node [right] at (5.7,3) {$\frac{\epsilon}{2}$};

%dashed circle
\draw [dashed] (5.1,1.9) circle [radius=1.1];
\node [right] at (5.1,1.9) {$p$};
\draw [fill] (5.1,1.9) circle [radius=0.5pt];
%arrow
\draw [<->] (5.05,1.85) -- (4.32,1.12);
\node [right] at (4.6,1.2) {$\frac{\epsilon}{2}$};

%geodesic
\node [right] at (4.1,2.3) {$p^\eta$};
\draw [pencildraw,rounded corners] [->] (5.1,1.9)  -- (4.6,2.0) -- (4.2,1.9);

\end{tikzpicture}
\caption{The chart $(x,U)$ with closed ball $G$ and open ball $O$ around $p_j$. As is shown in lemma~\ref{bndlem}, $p^\eta=p(t,\eta)$ does not leave the ball around $p$ with radius $\epsilon/2$, as long as $|t|\leq1/N$ for $N\geq N_\epsilon$. The importance for uniformity is that it does not matter where we choose $p$ (in $O$).\label{proof_charts}}
\end{centering}
\end{figure}

\textbf{Taylor expansion}\\
Now fix $p\in O$ and $\eta\in T_pM$. Write $p^\eta$ for the map $\R\rightarrow M$ that takes $t$ to $p(t,\eta)$. We can locally write $f\circ p^\eta = (f\circ x^{-1})\circ (x \circ p^\eta)$, which is a composition of smooth maps. This means that $f\circ p^\eta$ is just a smooth map $\R\rightarrow \R$, so we can use a Taylor expansion and obtain
\begin{equation*}
    f(p(1/N,\eta))=f(p)+\frac{1}{N} \frac{\dd (f\circ p^\eta)}{\dd t}(0)+\frac{1}{2N^2} \frac{\dd^2 (f\circ p^\eta)}{\dd^2 t}(0) + \frac{1}{6N^3} \frac{\dd^3 (f\circ p^\eta)}{\dd^3 t}(t_{N,\eta,p}),
\end{equation*}
where $t_{N,\eta,p}\in(0,1/N)$ is a number depending on $N$, $\eta$ and $p$. This gives us
\begin{eqnarray}
    N^2 L_Nf(p)&=& N^2 \int_{M_p} f(p(1/N,\eta))-f(p) \mu_p(\dd \eta)\nonumber \\
    &=& N^2 \int \frac{1}{N} \frac{\dd (f\circ p^\eta)}{\dd t}(0)+\frac{1}{2N^2} \frac{\dd^2 (f\circ p^\eta)}{\dd^2 t}(0) + \frac{1}{6N^3} \frac{\dd^3 (f\circ p^\eta)}{\dd^3 t}(t_{N,\eta,p}) \mu_p(\dd \eta)\nonumber \\
    &=& N \int \frac{\dd (f\circ p^\eta)}{\dd t}(0)\mu_p(\dd \eta)+\frac{1}{2}\int \frac{\dd^2 (f\circ p^\eta)}{\dd t^2}(0)\mu_p(\dd\eta) + \frac{1}{6N} \int  \frac{\dd^3 (f\circ p^\eta)}{\dd t^3}(t_{N,\eta,p}) \mu_p(\dd \eta).\hspace{0.7cm} \label{N2LN}
\end{eqnarray}
We will examine these terms separately. \\
\\
\textbf{The first term}\\
Recall that $p\in O$ and that $O$ is contained in a coordinate chart $(x,U)$. Since $N\geq N_\epsilon$, lemma~\ref{bndlem} guarantees us that $p(s,\eta)$ stays in the coordinate chart for $|s|<\frac{1}{N}$. Writing $\eta=\sum_{i=1}^n \eta^i \frac{\partial}{\partial x^i}|_p$, we see for $|s|<\frac{1}{N}$:
\begin{eqnarray*}
    \frac{\dd (f\circ p^\eta)}{\dd t}(s)&=&\frac{\dd}{\dd t} [(f\circ x^{-1}) \circ (x \circ p^\eta)](s)\\
    &=& \sum_{i=1}^n D_i(f\circ x^{-1})(x(p^\eta(s)) \frac{\dd (x^i \circ p^\eta)}{\dd t}(s)\\
    &=& \sum_{i=1}^n \frac{\partial f}{\partial x^i}( p^\eta(s)) \frac{\dd (x^i \circ p^\eta)}{\dd t}(s).
\end{eqnarray*}
Now setting $s=0$, this becomes:
\begin{eqnarray*}
    \sum_{i=1}^n \frac{\partial f}{\partial x^i}(p) \eta^i =  \sum_{i=1}^n \eta^i \frac{\partial}{\partial x^i}|_p f = \eta(f),
\end{eqnarray*}
since $p^\eta(0)=p(0,\eta)=p$ and the tangent vector to the geodesic $p(\cdot,\eta)$ at $0$ is $\eta$ (so the $i^{\text{th}}$ coordinate with respect $x$ is just $\eta^i$). Now the first term of~(\ref{N2LN}) becomes:
\begin{equation*}
    N\int \eta(f)\mu_p(\dd \eta) = N \int \sum_{i=1}^n \eta^i \frac{\partial}{\partial x^i}|_p f \mu_p(\dd \eta) = N  \sum_{i=1}^n \frac{\partial}{\partial x^i}|_p f \int \eta^i  \mu_p(\dd \eta).
\end{equation*}
By assumption these integrals are $0$. This shows that the first term of~(\ref{N2LN}) vanishes.\\
\\
%if we make the following assumption.\\
%\textit{\textbf{Assumption 2:} For every coordinate chart around any $p\in M$: $\int %\eta^i  \mu_p(\dd \eta)=0$ for each $i$.}\\
\textbf{The second term}\\
Now we want to show that the remaining term equals $\frac{1}{2}\Delta_Mf(p)$. Similarly to above we see for $|s|<\frac{1}{N}$ (leaving out the arguments to keep things clear):
\begin{eqnarray*}
    \frac{\dd^2 (f\circ p^\eta)}{\dd t^2} &=& \frac{\dd}{\dd t} \sum_{i=1}^n \frac{\partial f}{\partial x^i} \frac{\dd (x^i \circ p^\eta)}{\dd t}\\
    &=& \sum_{i=1}^n \left\{ \left(\frac{\dd}{\dd t} \frac{\partial f}{\partial x^i}\right) \frac{\dd (x^i \circ p^\eta)}{\dd t} +
    \frac{\partial f}{\partial x^i} \left(\frac{\dd}{\dd t} \frac{\dd (x^i \circ p^\eta)}{\dd t}\right) \right\}\\
    &=& \sum_{i=1}^n \left\{ \sum_{j=1}^n \frac{\partial^2 f}{\partial x^j\partial x^i} \frac{\dd (x^j \circ p^\eta)}{\dd t} \frac{\dd (x^i \circ p^\eta)}{\dd t} +
    \frac{\partial f}{\partial x^i} \frac{\dd^2 (x^i \circ p^\eta)}{\dd t^2} \right\}.
\end{eqnarray*}
Since $p^\eta$ is a geodesic, we know that it satisfies the geodesic equations. This shows that for each $i=1,..,n$ we have
\begin{equation*}
    \frac{\dd^2 (x^i \circ p^\eta)}{\dd t^2} + \sum_{k,l=1}^n \Gamma^i_{kl} \frac{\dd (x^k\circ p^\eta)}{\dd t}\frac{\dd (x^l\circ p^\eta)}{\dd t}=0.
\end{equation*}
Using this yields the following expression for the second derivative:
\begin{equation*}
    \sum_{i=1}^n \left\{ \sum_{j=1}^n \frac{\partial^2 f}{\partial x^j\partial x^i} \frac{\dd (x^j \circ p^\eta)}{\dd t} \frac{\dd (x^i \circ p^\eta)}{\dd t} -
    \frac{\partial f}{\partial x^i} \sum_{k,l=1}^n \Gamma^i_{kl} \frac{\dd (x^k\circ p^\eta)}{\dd t}\frac{\dd (x^l\circ p^\eta)}{\dd t} \right\},
\end{equation*}
so
\begin{equation*}
    \frac{\dd^2 (f\circ p^\eta)}{\dd t^2}(0)=\sum_{i=1}^n \left\{ \sum_{j=1}^n \frac{\partial^2 f}{\partial x^j\partial x^i}(p) \eta^j\eta^i -
    \frac{\partial f}{\partial x^i}(p) \sum_{k,l=1}^n \Gamma^i_{kl}(p) \eta^k\eta^l \right\}.
\end{equation*}
Using linearity of the integral, we obtain the following expression for the second term of~(\ref{N2LN}):
\begin{equation*}
    \frac{1}{2}\sum_{i=1}^n \left\{ \sum_{j=1}^n \frac{\partial^2 f}{\partial x^i\partial x^j}(p) \int\eta^i\eta^j\mu_p(\dd \eta) -
    \frac{\partial f}{\partial x^i}(p) \sum_{k,l=1}^n \Gamma^i_{kl}(p) \int \eta^k\eta^l\mu_p(\dd \eta) \right\}.
\end{equation*}
Note that we also changed the order of the derivatives of $f$, this can be done since $f$ is smooth. Now we want the term above to equal
\begin{eqnarray*}
    &&\frac{1}{2}\Delta_Mf(p)=\frac{1}{2}\left\{g^{ij}\frac{\partial^2 f}{\partial x^i x^j} - g^{kl}\Gamma^i_{kl} \frac{\partial f}{\partial x^i}\right\}\\
    &=& \frac{1}{2}\sum_{i=1}^n \left\{ \sum_{j=1}^n \frac{\partial^2 f}{\partial x^i\partial x^j}(p) g^{ij}(p) -
    \frac{\partial f}{\partial x^i}(p) \sum_{k,l=1}^n \Gamma^i_{kl}(p) g^{kl}(p) \right\}.
\end{eqnarray*}
This is true, since we required that for any coordinate chart around $p$ and for all $i,j$: $\int_{M_p}\eta^i\eta^j\mu_p(\dd \eta)=g^{ij}(p)$.\\
\\
\textbf{The rest term}\\
If the last term goes to $0$ uniformly on $O$, we have the result. Let $N$ still be larger then $N_\epsilon$.
\begin{eqnarray*}
    \left|\frac{1}{6N} \int  \frac{\dd^3 (f\circ p^\eta)}{\dd t^3}(t_{N,\eta,p}) \mu_p(\dd \eta)\right|
    &\leq& \frac{1}{6N} \int \left| \frac{\dd^3 (f\circ p^\eta)}{\dd t^3}(t_{N,\eta,p})\right| \mu_p(\dd \eta) \\
    &\leq&  \frac{K'}{6N} \sup_{\eta\in\supp\mu_p} \left| \frac{\dd^3 (f\circ p^\eta)}{\dd t^3}(t_{N,\eta,p})\right|
\end{eqnarray*}
where $K'=\sup_{p\in M} \mu_p(T_pM)<\infty$ (by assumption). We know that $t_{N,\eta,p}\in [0,1/N]\subset [0,1/N_\epsilon]$. This means that the above is smaller than:
\begin{equation*}
    \frac{K'}{6N} \sup_{\eta\in\supp\mu_p}\sup_{t\in [0,1/N_\epsilon]} \left| \frac{\dd^3 (f\circ p^\eta)}{\dd t^3}(t)\right|
    \leq \frac{K'}{6N} \sup_{\eta:||\eta||\leq K}\sup_{t\in [0,1/N_\epsilon]} \left| \frac{\dd^3 (f\circ p^\eta)}{\dd t^3}(t)\right|.
\end{equation*}
Because of the $1/N$ in front of the equation, we only need to know that the rest is uniformly bounded to obtain uniform convergence. It thus suffices to show that $\frac{\dd^3 (f\circ p^\eta)}{\dd t^3}(t)$ is bounded as a function of $\eta$ with $||\eta||<K$ and $t\in[0,1/N_\epsilon]$. Lemma~\ref{bndlem} shows that $p(t,\eta)$ stays in $G$ for all such $\eta$ and $t$. We will use this fact multiple times.\\
\\
We first express $\frac{\dd^3 (f\circ p^\eta)}{\dd t^3}$ in local coordinates for $|t|\leq 1/N$.
\begin{equation}
    \frac{\dd^3 (f\circ p^\eta)}{\dd t^3} = \frac{\dd}{\dd t}\frac{\dd^2 (f\circ p^\eta)}{\dd t^2} = \frac{\dd}{\dd t}
    \sum_{i=1}^n \left\{ \sum_{j=1}^n \frac{\partial^2 f}{\partial x^j\partial x^i} \frac{\dd (x^j \circ p^\eta)}{\dd t} \frac{\dd (x^i \circ p^\eta)}{\dd t} +
    \frac{\partial f}{\partial x^i} \frac{\dd^2 (x^i \circ p^\eta)}{\dd t^2} \right\}\label{thirdder}.
\end{equation}
To make notation more compact, we introduce the following notation (and $f_i,f_{ijk}$ analogously):
\begin{equation*}
    f_{ij}:=\frac{\partial^2 f}{\partial x^j\partial x^i},\hspace{2cm} p^i_k:=\frac{\dd^k (x^i \circ p^\eta)}{\dd t^k}.
\end{equation*}
Combining this with Einstein summation, we can write~(\ref{thirdder}) as
\begin{eqnarray*}
    \frac{\dd}{\dd t} (f_{ij}p^i_1p^j_1 + f_ip^i_2) &=& (f_{ijk}p^k_1)p^i_1p^j_1+f_{ij}(p^i_1p^j_2+p^i_2p^j_1) + (f_{ij}p^j_1)p^i_2+f_ip^i_3\\
    &=& f_{ijk}p^k_1p^i_1p^j_1+f_{ij}(p^i_1p^j_2+2p^i_2p^j_1) +f_ip^i_3.
\end{eqnarray*}
Now, as before, we can deal with second derivatives of geodesics using the geodesic equations:
\begin{equation*}
    p^i_2=-\Gamma^i_{rs}p^r_1p^s_1.
\end{equation*}
We can also calculate the third derivative:
\begin{equation*}
    p^i_3=\frac{\dd}{\dd t}p^i_2=\frac{\dd}{\dd t}(-\Gamma^i_{rs}p^r_1p^s_1) = -\left(\frac{\dd}{\dd t}\Gamma^i_{rs}\right)p^r_1p^s_1-\Gamma^i_{rs}(p^r_1p^s_2+p^r_2p^s_1).
\end{equation*}
This shows us that $\frac{\dd^3 (f\circ p^\eta)}{\dd t^3}$ is a combination of products and sums of the following types of expressions: $f_i$, $f_{ij}$, $f_{ijk}$, $p^i_1$, $\Gamma^i_{rs}$ and $\frac{\dd}{\dd t}\Gamma^i_{rs}$. If we can bound all of these on the right domains (independent of $p$ and $\eta$), we are done.\\
\\
\textbf{Bounding $f_i$, $f_{ij}$ and $f_{ijk}$}\\
First of all, note that $f$ is a smooth function on $U$. Further, $\partial_i$ defines smooth vector field on $U$. Since $f_i=\frac{\partial f}{\partial x^i}$ is obtained by applying $\partial_i$ on $U$ to $f$, it is a smooth function on $U$. Continuing in this way, we see that $f_{ij}$ and $f_{ijk}$ are also smooth functions on $U$. In particular, they are smooth functions on $G$ (since it is a subset of $U$). $G$ is a closed subset of the compact $M$ and is hence compact itself. This implies that $f_i$, $f_{ij}$ and $f_{ijk}$ are (for each choice of $i,j,k$) bounded on $G$. Since we evaluate these functions in the points $p(s,\eta)$ for $0\leq s\leq 1/N$, $N\geq N_\epsilon$ and $||\mu||\leq K$, our discussion above shows that we only evaluate them in points of $G$. This means that we have found bounds for $f_i$, $f_{ij}$ and $f_{ijk}$.\\
\\
\textbf{Bounding $p^i_1$}\\
We start with a technical lemma.
\begin{lemma}\label{valofcomp}
Let $q\in M$ and let $(y,V)$ be a coordinate chart around $q$. Let $v\in T_qM$ and write $v=v^i\partial_i$. Then $|v^i|\leq \sqrt{g^{ii}(q)}||v||$.
\end{lemma}
\begin{proof}
Fix some $1\leq i\leq n$. We see in the tangent space at $q$:
\begin{equation*}
    \left<v,g^{ij}\partial_j\right>=\left<v^k\partial_k,g^{ij}\partial_j\right>=v^kg^{ij}g_{kj}=v^k\delta^i_k=v^i.
\end{equation*}
Further,
\begin{equation*}
    ||g^{ij}\partial_j||^2=\left<g^{ij}\partial_j,g^{ik}\partial_k\right>=g^{ij}g^{ik}g_{jk}=g^{ij}\delta^i_j=g^{ii}.
\end{equation*}
Using the relations above and the Cauchy-Schwarz inequality, we obtain:
\begin{equation*}
    |v^i|=|\left<v,g^{ij}\partial_j\right>|\leq ||v||\cdot||g^{ij}\partial_j||=\sqrt{g^{ii}}||v||.
\end{equation*}
\end{proof}
Now we can use this to show the following.
\begin{lemma}
$|p^i_1(t)|=\left|\frac{\dd (x^i\circ p^\eta)}{\dd t}(t)\right| \leq \sqrt{g^{ii}(p(t,\eta))}||\eta||$.
\end{lemma}
\begin{proof}
The first equation is just a change of notation. Further we see
\begin{equation*}
    \frac{\dd (x^i\circ p^\eta)}{\dd t}=\left(p^\eta_* \frac{\dd}{\dd t}\right)(x^i)=\frac{\dd p^\eta}{\dd t}(x^i)=\left(\frac{\dd p^\eta}{\dd t}\right)^i.
\end{equation*}
This means that $\frac{\dd (x^i\circ p^\eta)}{\dd t}$ is just the $i^{\text{th}}$ coordinate with respect to $(x,U)$ of the tangent vector to $p^\eta$ at time $t$ so at the point $p(t,\eta)\in M$. Using lemma~\ref{valofcomp}, we see
\begin{equation}\label{eqhoi}
    \left|\frac{\dd (x^i\circ p^\eta)}{\dd t}(t)\right| \leq \sqrt{g^{ii}(p(t,\eta))}\left|\left|\frac{\dd p^\eta}{\dd t}\right|\right|.
\end{equation}
Since $p^\eta$ is a geodesic, it has constant speed. Its speed at $p$ is $||\eta||$, so this must be its speed anywhere else along the trajectory. Hence $||\frac{\dd p^\eta}{\dd t}||=||\eta||$. Inserting this in~(\ref{eqhoi}) yields the result.
\end{proof}
We can now easily obtain a bound for $p^i_1$. For $0\leq t\leq 1/N$ and $||\eta||\leq K$, we know $p(t,\eta)$ stays in $G$. $g^{ii}$ is a smooth and hence continous function on $U$, so it is bounded on $G$ (since $G$ is compact). This means that $\sqrt{g^{ii}(p(t,\eta))}$ is bounded by some $K^{i}$ for $||\eta||\leq K$ and $0\leq t\leq 1/N$. Now we see $|p^i_1|\leq \sqrt{g^{ii}(p(t,\eta))}\left|\left|\frac{\dd p^\eta}{\dd t}\right|\right|\leq K^{i}K$.\\
\\
\textbf{Bounding $\Gamma^i_{rs}$ and $\frac{\dd}{\dd t}\Gamma^i_{rs}$}\\
Each $g_{ij}$ is a smooth function on $U$. This means that $\frac{\partial g_{ij}}{\partial x^k}$ is a smooth function on $U$. This implies that $\Gamma^i_{rs}$ is just combination of products and sums of smooth functions, so it is smooth itself. Now, as before, $\Gamma^i_{rs}$ is bounded on $G$. Since we only evaluate it in $p(t,\eta)$ with $0\leq t\leq 1/N$ and $||\eta||\leq K$, we only evaluate it in $G$, so we have bounded $\Gamma^i_{rs}$. \\
Now $\frac{\dd}{\dd t}\Gamma^i_{rs}$ can be written as
\begin{equation*}
    \frac{\dd}{\dd t}\Gamma^i_{rs}=\frac{\partial \Gamma^i_{rs}}{\partial x^j} \frac{\dd (x^j\circ p^\eta)}{\dd t} = (\Gamma^i_{rs})_j p^j_1,
\end{equation*}
with notation as above. Since $\Gamma^i_{rs}$ is smooth function $U\rightarrow \R$, this expression can be bounded in exactly the same way as expressions like $f_jp^j_1$ above. 
\subsection{Stepping distribution}\label{3stepdist}
\textbf{Constraints for a stepping distribution}\\
The question now is which distributions $\mu_p$ on $T_pM$ satisfy the assumptions of proposition~\ref{genconv}. From here on we fix $p\in M$ and simply write $\mu$ for $\mu_p$. Being compactly supported and finite are rather natural constraints, but the other assumptions are harder, especially since they involve local coordinates. In this section we address the question which distributions satisfy the other assumptions, i.e. for every coordinate system around $p$:
\begin{equation}\label{symreqorig}
    \begin{alignedat}{3}
        &\int \eta^i \mu(\dd \eta) &&= 0 \hspace{1cm} &&\forall i=1,..,n\\
        &\int \eta^i\eta^j \mu(\dd \eta) &&= g^{ij}   &&\forall i,j=1,..,n.
    \end{alignedat}
\end{equation}

To generalize this a bit, suppose $\mu$ satisfies the following for some $c>0$ for every coordinate system:
\begin{equation}\label{symreq}
    \begin{alignedat}{3}
        &\int \eta^i \mu(\dd \eta) &&= 0 \hspace{1cm} &&\forall i=1,..,n\\
        &\int \eta^i\eta^j \mu(\dd \eta) &&= cg^{ij}   &&\forall i,j=1,..,n.
    \end{alignedat}
\end{equation}
Following the proof in the previous section, one sees directly that in this case the generators converge to the generator of Brownian motion that is speeded up by a factor $c$. We will look into this generalized situation and at the end we will see how to determine $c$.\\
\\
\textbf{Independence of~(\ref{symreq}) of coordinate systems}\\
The following lemma shows that if~(\ref{symreq}) holds for a single coordinate system, it holds for any coordinate system.
\begin{lemma}\label{onesystemsuffices}
If~(\ref{symreq}) holds for some $c>0$ and for some coordinate system $(x,U)$ around $p$, then it holds for the same $c$ for all coordinate systems around $p$.
\end{lemma}
\begin{proof}
Let $(x,U)$ be a coordinate system around $p$ for which~(\ref{symreq}) holds with $c>0$ and let $(y,V)$ be any other coordinate system around $p$. It suffices to show that~(\ref{symreq}) holds with the same $c$ for $y$. Denote the metric matrix with respect to $x$ by $g$ and the one with respect to $y$ by $\hat g$. For any $\eta\in T_pM$ define $\eta^1,..,\eta^n$ as the coefficients of $\eta$ with respect to $x$, so such that $\eta=\sum_i \eta^i \frac{\partial}{\partial x^i}$. Analogously let $\hat\eta^1,..,\hat\eta^n$ be such that $\eta=\sum_i \hat\eta^i \frac{\partial}{\partial y^i}$. Let $J=\frac{\partial (x^1,..,x^n)}{\partial (y^1,..,y^n)}$. If $\eta\in T_pM$, then
\begin{equation*}
    \hat\eta^j=\eta(y^i)=\sum_i \eta^i \frac{\partial}{\partial x^i} y^i = \sum_i \eta^i \frac{\partial y^j}{\partial x^i}.
\end{equation*}
This shows that for any $j$
\begin{equation*}
    \int\hat\eta^j\mu(\dd \eta) = \int \sum_{i=1}^n \eta^i \frac{\partial y^j}{\partial x^i} \mu(\dd \eta) = \sum_{i=1}^n \frac{\partial y^j}{\partial x^i} \int \eta^i \mu(\dd \eta) = 0, 
\end{equation*}
since for any $i$: $\int \eta^i \mu(\dd \eta) = 0$. Moreover, for any $i,j$: $\int \eta^i\eta^j\mu(\dd \eta)=c g^{ij}$, so for any $i,j$:
\begin{eqnarray*}
    \int \hat\eta^i\hat\eta^j\mu(\dd \eta) &=& \int \sum_{k=1}^n \eta^k \frac{\partial y^i}{\partial x^k} \sum_{l=1}^n \eta^l \frac{\partial y^j}{\partial x^l} \mu(\dd \eta) = \sum_{k,l=1}^n \frac{\partial y^i}{\partial x^k}\frac{\partial y^j}{\partial x^l} \int \eta^k \eta^l \mu(\dd \eta)\\
    &=& \sum_{k,l=1}^n \frac{\partial y^i}{\partial x^k}\frac{\partial y^j}{\partial x^l} c g^{kl} = c (J^{-1}G^{-1}(J^{-1})^T)_{ij}.
\end{eqnarray*}
Since $J^{-1}G^{-1}(J^{-1})^T=J^{-1}G^{-1}(J^{T})^{-1}=(J^TGJ)^{-1}=\hat G ^{-1}$, we see that $\int \hat\eta^i\hat\eta^j\mu(\dd \eta)=c\hat g^{ij}$. We conclude that~(\ref{symreq}) holds for $y$ with the same $c$.
\end{proof}
\sk
\textbf{Orthogonal transformations and canonical measures}\\
We now introduce a class of measures.
\begin{defi}
Let $V$ be an inner product space and let $T$ be a linear map $V\rightarrow V$. We call $T$ an \textit{orthogonal transformation} if for any $u,v\in V$: $\left<Tu,Tv\right>=\left<u,v\right>$.\\
We call a measure $\mu$ on $T_pM$ \textit{canonical} if for any orthogonal transformation $T$ on $T_pM$ and for any coordinate system:
\begin{equation*}
    \int \eta^i \mu(\dd \eta) = \int (T\eta)^i \mu(\dd \eta) \text{ and } \int \eta^i \eta^j \mu(\dd \eta) = \int (T\eta)^i(T\eta)^j \mu(\dd \eta).
\end{equation*}
\end{defi}
\begin{rmk}\label{zeromean}In the same way as above, one can show that $\mu$ has the property above with respect to some coordinate system if and only if it has the property with respect to every coordinate system. Moreover, since $-I$ always satisfies $(-I)^TG(-I)=G$, we see that $\int \eta^i \mu(\dd \eta)=\int (-\eta)^i \mu(\dd \eta)=\int -\eta^i \mu(\dd \eta)=-\int \eta^i \mu(\dd \eta)$, so $\int \eta^i \mu(\dd \eta)$ is $0$ for any canonical $\mu$.
\end{rmk}
In words, $\mu$ is canonical if orthogonal transformations do not change the mean vector and the covariance matrix of a random variable that has distribution $\mu$. Remark~\ref{zeromean} shows that in fact the mean vector must be $0$. Note that in particular measures that are invariant under orthogonal transformations are canonical, since then $\int (T\eta)^i \mu(\dd \eta) = \int \eta^i (\mu\circ T^{-1})(\dd\eta) = \int \eta^i\mu(\dd \eta)$ and the other equation follows analogously. However a simple example shows that the converse is not true.
Let $M=\R$ and let $\mu$ be any non-symmetric distribution on $T_pM=\R$ with mean $0$. The only orthogonal transformation (apart from the identity) is $t\mapsto -t$. Under this transformation the mean (which is $0$) and the second moment are obviously left invariant, but $\mu$ is not symmetric, so it is not invariant. We will give an example for $\R^n$ later.\\
\\
If $(x,U)$ is some coordinate system around $p$ and $G=(g_{ij})$ is the matrix of the metric in $p$ with respect to $x$, we can write a linear transformation $T:T_pM\rightarrow T_pM$ as a matrix (which we will also call $T$) with respect to the base $\frac{\partial}{\partial x^1},..,\frac{\partial}{\partial x^n}$. We see that
\begin{eqnarray*}
    \left<T\eta,T\xi\right>=\sum_{i,j} g_{ij} (T\eta)^i(T\xi)^j = \sum_{i,j} g_{ij} \sum_{k} T_{ik} \eta^k \sum_{l} T_{jl} \xi^l=\sum_{k,l} \left(\sum_{i,j} g_{ij} T_{ik}T_{jl}\right) \eta^k \xi^l.
\end{eqnarray*}
If $T$ is orthogonal, this must equal
\begin{eqnarray*}
    \left<\eta,v\right>=\sum_{k,l} g_{kl} \eta^k\xi^l,
\end{eqnarray*}
so we see that $g_{kl}=\sum_{i,j} g_{ij} T_{ik}T_{jl}=(T^TGT)_{kl}$ and hence $G=T^TGT$.\\
Now for a measure $\mu$ on $T_pM$ and a coordinate system $(x,U)$, define the vector $A_\mu$ and the matrix $B_\mu$ by $A_\mu^i=\int \eta^i\mu(\dd\eta)$ and $B_\mu^{ij}=\int\eta^i\eta^j\mu(\dd \eta)$. Then we have the following.
\begin{lemma}\label{canonmatrix}
Let $\mu$ be a measure on $T_pM$. Then the following are equivalent.
\begin{enumerate}[(i)]
    \item $\mu$ is canonical.
    \item For every linear transformation $T$ and every coordinate system $(x,U)$: if $G=T^TGT$ , then $A_\mu=TA_\mu$ and $B_{\mu}=TB_\mu T^T$.
    %\item For every linear transformation $T$: if there is a coordinate system such that $G=T^TGT$, then there is a coordinate system such that $A_\mu=TA_\mu$ and $B_{\mu}=TB_\mu T^T$.
\end{enumerate}
\end{lemma}
\begin{proof}
$(i)\Leftrightarrow (ii)$ because $(ii)$ is just the definition of being canonical written in local coordinates. Indeed, we already saw that orthogonality or $T$ translates in local coordinates to $G=T^TGT$, the other expressions follow in a similar way from the following equations:
\begin{eqnarray*}
    A_\mu^{i}=\int (T\eta)^i \mu(\dd \eta) &=& \int \sum_k T_{ik} \eta^k\mu(\dd \eta) = \sum_k T_{ik} \int \eta^k\mu(\dd \eta) = \sum_k T_{ik}A_\mu^k\\
    B_\mu^{ij}=\int (T\eta)^i(T\eta)^j\mu(\dd \eta)&=&\int \sum_k T_{ik} \eta^k \sum_l T_{jl} \eta^l \mu(\dd \eta) = \sum_{k,l} T_{ik} T_{jl} \int \eta^k\eta^l\mu(\dd \eta) = \sum_{k,l} T_{ik} T_{jl} B_\mu^{kl}.
\end{eqnarray*}
%$(iii)$ follows from $(ii)$ a fortiori. 
\end{proof}
\sk

\textbf{Canonical measures are stepping distributions}\\
Now we have the following result.
\begin{prop}
Let $\mu$ be a probability measure on $T_pM$. Then $\mu$ is canonical if and only if it satisfies (\ref{symreq}) for some $c>0$.
\end{prop}
\begin{proof}
First assume that $\mu$ is canonical and let $(x,U)$ be normal coordinates centered at $p$. Because of lemma~\ref{onesystemsuffices} it suffices to verify~(\ref{symreq}) for $x$, so we need to show that $A_\mu=0$ and $B_\mu=cG^{-1}=cI$ for some $c>0$. \\
The fact that $A_\mu=0$ is just remark~\ref{zeromean}. Now note that since $B_\mu$ is symmetric, it can be diagonalized as $TB_\mu T^{-1}$ where $T$ is an orthogonal matrix (in the usual sense). This means that $T^T=T^{-1}$ and that $T^TGT=T^TIT=T^TT=I=G$, so lemma~\ref{canonmatrix} tells us that the diagonalization equals $TB_\mu T^T=B_\mu$. This implies that $B_\mu$ is a diagonal matrix. Now for $i\neq j$ let $\bar I^{ij}$ be the $n\times n$-identity matrix with the $i^\text{th}$ and $j^\text{th}$ column exchanged. It is easy to see that $(\bar I^{ij})^T\bar I^{ij}=I$, so we must also have $B_\mu=\bar I^{ij}B_\mu(\bar I^{ij})^T$. The latter is $B_\mu$ with the $i^\text{th}$ and $j^\text{th}$ diagonal element exchanged. This shows that these elements must be equal. Hence all diagonal elements are equal and $B_\mu=cI$ for some $c\in\R$. Since $c=B_\mu^{11}=\int \eta^1\eta^1\mu(\dd\eta)\geq 0$, we know that $c\geq 0$. If $c=0$, then $B_\mu=0$, so $\mu=0$, which is not possible. We conclude that $c>0$.\\
Conversely let $(x,U)$ be a coordinate system with corresponding metric matrix $G$ and assume that $\mu$ satisfies~(\ref{symreq}) for some $c>0$. Let $T$ be such that $G=T^TGT$. Then $A_\mu=0=T0=TA_\mu$. We also see: $T^TGT=G \iff G=(T^T)^{-1}GT^{-1} \iff G^{-1}=TG^{-1}T^T \iff cG^{-1}=T(cG^{-1})T^T \implies B_\mu=TB_\mu T^T$ (since $B_\mu=cG^{-1}$), so by lemma~\ref{canonmatrix} $\mu$ is canonical.
\end{proof}

Now we know that if the stepping distribution is canonical (and finite and compactly supported, uniformly on $M$), the generators converge to the generator of Brownian motion that is speeded up by some factor $c>0$ (depending on $\mu$). The question remains what this $c$ is. The following lemma answers this question.
\begin{lemma}
Suppose $\mu$ satisfies (\ref{symreq}) for some $c>0$. Then $c=\frac{\int ||\eta||^2\mu(\dd \eta)}{n}$.
\end{lemma}
\begin{proof}
We calculate the following (with respect to some coordinate system $(x,U)$):
\begin{eqnarray*}
    \int ||\eta||^2 \mu(\dd \eta) &=& \int \left<\eta,\eta\right>\mu(\dd \eta) = \int \left<\sum_i \eta^i\frac{\partial}{\partial x^i},\sum_j \eta^j\frac{\partial}{\partial x^j}\right>\mu(\dd \eta) \\
    &=& \sum_{i,j}\left<\frac{\partial}{\partial x^i},\frac{\partial}{\partial x^j}\right>\int \eta^i\eta^j\mu(\dd \eta)
    = \sum_{i,j} g_{ij}c g^{ij} = c \sum_i \sum_j g_{ij}g^{ji} = c\sum_i 1 = cn.
\end{eqnarray*}
Hence $c=\frac{\int ||\eta||^2\mu(\dd \eta)}{n}$.
\end{proof}
The nice part of this lemma is that the expression for $c$ does not involve a coordinate system, only the norm (and hence inner product) of $T_pM$. In particular we see that $c=1$ is equivalent to $\int ||\eta||^2\mu(\dd \eta)=n$.
We summarize our findings in the following result.
\begin{prop}\label{sumprop}
A probability measure $\mu$ on $T_pM$ satisfies (\ref{symreq}) for some $c>0$ if and only if it is canonical and $c=\frac{\int ||\eta||^2\mu(\dd \eta)}{n}$. In particular, it satisfies~(\ref{symreqorig}) %(and, if compactly supported, hence qualifies as a stepping distribution) 
if and only if it is canonical and $\int ||\eta||^2\mu(\dd \eta)=n$.
\end{prop}

\begin{rmk}\label{notiid}
Note that all we need of the jumping distributions is that their mean is 0, their covariance matrix is invariant under orthogonal transformations, they are (uniformly) compactly supported and they are (uniformly) finite. We don't need the measures to be similar in any other way, so we do not at all require the jumps to have identical distributions in the sense of~\cite{jorgensen1975central}.
\end{rmk}\sk

\textbf{Examples}\\
1. To satisfy~(\ref{symreqorig}) for every coordinate system, by lemma~\ref{onesystemsuffices} it suffices to choose a coordinate system and construct a distribution that satisfies~(\ref{symreqorig}) for that coordinate system. Let $(x,U)$ be any coordinate system around some point in $M$ with corresponding metric matrix $G$ in that point. Let $X$ be any random variable in $\R^n$ that has mean vector $0$ and covariance matrix $G^{-1}$ (for instance let $X\sim N(0,G^{-1})$). Now let $\mu$ be the distribution of $\sum_i X^i\frac{\partial}{\partial x^i}$. Then by construction $\int \eta^i\mu(\dd \eta)=\E X^i = 0$ and $\int \eta^i\eta^j\mu(\dd \eta)=\E X^iX^j = \E X^iX^j -\E X^i\E X^j=g^{ij}$.\\ %However, it is not compactly supported, so it still does not satisfy the requirements of proposition~\ref{genconv}.\\
2. In the previous example~(\ref{symreqorig}) is immediate. Let us now consider an example that illustrates the use of proposition~\ref{sumprop}. Let $\mu_p$ be the uniform distribution on $\sqrt{n}S_pM$ (the vectors with norm $\sqrt n$). By definition of such a distribution, it is invariant under orthogonal transformations (rotations and reflections), so it is a canonical distribution. Since also $\int ||\eta||^2\mu(\dd \eta) = \int \sqrt{n}^2 \mu(\dd \eta)=n$, we conclude that the uniform distribution on $\sqrt{n}S_pM$ satisfies~(\ref{symreqorig}). Moreover, $\sup_{p\in M} \sup_{\eta\in\supp \mu_p} ||\eta||=\sqrt(n)<\infty$ and $\sup_{p\in M} \mu_p(T_pM)=1<\infty$. Together this shows that the $\mu_p$'s satisfy the assumption of proposition~\ref{genconv}.\\
3. Let us conclude by showing for $\R^n$ that the class of canonical distributions is strictly larger than the class of distributions that are invariant under orthogonal transformations, even with the restriction that $\int ||\eta||^2\mu(\dd \eta)=n$. It suffices to find a distribution $\mu$ with mean $0$ and covariance matrix $I$ (since then $\mu$ satisfies~(\ref{symreqorig}) and~\ref{sumprop} then tells us that $\mu$ is canonical and has $\int ||\eta||^2\mu(\dd \eta)=n$) and an orthogonal $T$ such that $\mu\neq \mu\circ T^{-1}$. Let $\nu$ be the distribution on $\R$ given by $\nu=\frac{1}{5}\delta_{-2}+\frac{4}{5}\delta_{1/2}$. Then, using the natural coordinate system, $\int t \nu(\dd t)=\frac{1}{5}(-2)+\frac{4}{5}\frac{1}{2}=0$ and $\int t^2\mu(\dd t) = \frac{1}{5}(-2)^2+\frac{4}{5}(\frac{1}{2})^2=1$. Now let $\mu=\nu\times..\times\nu$ ($n$ times). Then we directly see that the mean vector is $0$ and the covariance matrix is $I$. However $T=-I$ is an orthogonal transformation and $\mu\circ (-I)^{-1}$ equals the product of $n$ times $\frac{1}{5}\delta_{2}+\frac{4}{5}\delta_{-1/2}$, so obviously $\mu\neq \mu\circ (-I)^{-1}$.
 %\newpage
\section{Uniformly approximating grids}\label{part2}
We would like to consider interacting particle systems such as the symmetric exclusion process on a manifold. Because the exclusion process does not make sense directly in a continuum, we need a proper discrete grid approximation. More precisely, we need a sequence of grids on the manifold that converges to the manifold in a suitable way. It will become clear that the grids will need to approximate the manifold in a uniform way. We will see in section~\ref{part1} that a natural requirement on the grids is that we can define edge weights (or, equivalently, random walks) on them, such that the graph Laplacians converge to the Laplace-Beltrami operator in a suitable sense. \\
To be more precise, we would like to have a sequence $(p_n)_{n=1}^\infty$ in $M$ and construct a sequence of grids $(G^N)_{N=1}^\infty$ by setting $G^N=\{p_1,..,p_N\}$. On each $G^N$, we would like to define a random walk $X^N$ which jumps from $p_i$ to $p_j$ with (symmetric) rate $W^N_{ij}$ with the property that there exists some function $a:\N\rightarrow[0,\infty)$ and some constant $C>0$ such that for each smooth $\phi$
\begin{equation*}
    a(N)\sum_{j=1}^NW^N_{ij}(\phi(p_j)-\phi(p_i))\longrightarrow C\Delta_M\phi(p_i)\quad(N\rightarrow\infty)
\end{equation*}
where the convergence is in the sense that for all smooth $\phi:M\rightarrow \R$
\begin{equation}\label{unifapproxgrid}
    \lim_{N\rightarrow\infty} \frac{1}{N}\sum_{i=1}^N \left|a(N)\sum_{j=1}^NW^N_{ij}(\phi(p_j)-\phi(p_i))- C\Delta_M\phi(p_i)\right|= 0.
\end{equation}
\begin{defi}
We call a sequence of grids and corresponding weights $(G_N,W_N)_{N=1}^\infty$ uniformly approximating grids if they satisfy~(\ref{unifapproxgrid}).
\end{defi}
\begin{rmk}[Comparison with standard grids]\label{comprmk}
To give an idea of how known grids in Euclidean spaces can be incorporated in this framework, let $S$ be the one-dimensional torus. Let $S^N$ be the grid that places a grid point in $k/N, k=1,..,N$.
%These grids differ from the ones described above, since here not only points are added, but everything is rescaled. However, if we consider these grids for $N=2^m$, then we can interpret $S^{2^{m+1}}$ as being obtained from $S^{2^m}$ by adding a point between any two neighbouring points.
Now we can define a nearest neighbour random walk by putting $W^N_{ij}=\1_{|p_i-p_j|=1/N}$. Also set $a(N)=N^2$. Then we see for a point $p_i\in S^N$ for $N=2^m$ for some $m\in\N$ that
\begin{equation*}
    a(N)\sum_{j=1}^NW^N_{ij}(\phi(p_j)-\phi(p_i)) = N^2 (\phi(p_i+1/N)+\phi(p_i-1/N)-2\phi(p_i)) = \phi''(p_i)+O(N^{-1}).
\end{equation*}
The compactness of the torus easily implies that this rest term can be bounded uniformly. This implies that~(\ref{unifapproxgrid}) holds.% (at least along the subsequence $N=2^m, m=0,1,2,..$).
\end{rmk}
We will show in section~\ref{part1} that if we define the Symmetric Exclusion Process on uniformly approximating grids we can prove that its hydrodynamic limit satisfies the heat equation on $M$. \\
\\
It is not obvious how uniformly approximating grids could be defined. Most natural grids in Euclidean settings involve some notion of equidistance, scaling or translation invariance. All of these concepts are very hard if not intrinsically impossible to define on a manifold. The current section is dedicated to showing that uniformly approximating grids actually exist. To be more precise, we will show that a sequence $(p_n)_{n=1}^\infty$ can be used to define such grids if the empirical measures $1/N\sum_{i=1}^N\delta_{p_i}$ converge to the uniform distribution in Kantorovich sense. In section~\ref{randomGrid} we will show that such sequences exist: they are obtained with probability 1 when sampling uniformly from the manifold, i.e. from the normalized Riemannian volume measure.\\
\\
For the calculations of this section, we need a result that forms the core of proving the invariance principle, which we have proved in section~\ref{part3}.\\
\\
\begin{rmk} At first sight the requirement that the empirical measures approximate the uniform measure and that the grid points can be sampled uniformly seems arbitrary, but this is actually quite natural. We want to construct a random walk with symmetric jumping rates (we need this for instance for the Symmetric Exclusion Process later). This implies that the invariant measure of the random walk is the counting measure, so the random walk spend on average the same amount of time in each point of the grid. Hence the amount of time that the random walk spends in some subset of the manifold is proportional to the amount of grid points in that subset. Since we want the random walk to approximate Brownian motion and the volume measure is invariant for Brownian motion, we want the amount of time that the random walk spends in a set to be proportional to the volume of the set. This means that the amount of grid points in a subset of $M$ should be proportional to the volume of that subset. This suggests that the empirical measures $1/N\sum_{i=1}^N\delta_{p_i}$ should in some sense approximate the uniform measure. Moreover, a natural way to let the amount of grid points in a subset be proportional to its volume is by sampling grid points from the uniform distribution on the manifold.
\end{rmk}

\subsection{Model and motivation}\label{1modmot}
\textbf{Motivation}\\
In some areas of statistics the following is known and used (see for instance~\cite{singer2006graph}). Suppose we have a manifold $M$ that is imbedded in $\R^m$ for some $m$ and we would like to recover the manifold from some observations of it, say an i.i.d. sample of uniform random elements of $M$. To do this we can describe the observations as a graph with as weight on the edge between two points a semi positive kernel with bandwidth $\epsilon$ applied to the Euclidean distance between those points. Then it can be shown that the graph Laplacian of the graph that is obtained in this way converges in a suitable sense to the Laplace-Beltrami operator on $M$ as the number of observations goes to infinity and $\epsilon$ goes to $0$. This suggests that we could define random walks on such random graphs and that the corresponding generators converge to the generator of Brownian motion. We generalize this idea by taking a more general sequence of graphs, but our main example (in section~\ref{randomGrid}) will be this random graph.\\
\\
A main point of concern is the following: we prefer to view the manifold $M$ on its own instead of imbedded in a Euclidean space. This means that we would like to use the distance that is induced by the Riemannian metric instead of the Euclidean distance. The latter is more suitable to purposes in statistics, because in that setting the Riemannian metric on $M$ is not known beforehand. Also, a lot is known about the behaviour of the Euclidean distance in this type of situation and not so much about the distance on the manifold. We will have to make things work in $M$ itself.\\
\\
\textbf{Model}\\
Let $M$ be a compact and connected Riemannian manifold. We call a function $f$ on $M$ Lipschitz with Lipschitz constant $L_f$ if
\begin{equation*}
    \sup_{p,q\in M}\frac{|f(p)-f(q)|}{d(p,q)} = L_f<\infty.
\end{equation*}
Let $(p_n)_{n\geq 1}$ be a sequence in $M$ such that $\mu^N:=\frac{1}{N}\sum_{i=1}^N\delta_{p_i}$ converges in the Kantorovich sense to $\bar V$ (the uniform distribution on $M$), i.e.
\begin{equation*}
    W_1(\mu^N,\bar V) = \sup_{f\in\mathcal{F}_1(M)} \left\{\int_M f\dd\mu^N - \int_M f \dd \bar V\right\} \rightarrow 0,
\end{equation*}
where $\mathcal{F}_1(M)$ denotes the set of Lipschitz functions $f$ on $M$ that have Lipschitz constant $L_f\leq 1$.
Define the $N^{\text{th}}$ grid $V_N$ as $V_N=\{p_1,..,p_N\}$. Set 
\begin{equation}\label{defepsilon}
    \epsilon:=\epsilon(N):=\left(\sup_{m\geq N} W_1(\mu^m,\bar V)\right)^{\frac{1}{4+d}}.
\end{equation} 
This $\epsilon$ rescales the distance over which particles will jump. Naturally, $\epsilon\downarrow 0$ as $N\rightarrow \infty$ (since $W_1(\mu^N,\bar V)\rightarrow 0$). 
%be some semi positive kernel $[0,\infty)\rightarrow[0,\infty)$, i.e. $k$ is non-negative and decreasing. 
Let $k:[0,\infty)\rightarrow[0,\infty)$ be Lipschitz and compactly supported (for instance $k(x)=(1-x)\1_{[0,1]}(x)$), we will call such $k$ a kernel. Define 
\begin{equation*}
    W^\epsilon_{ij}=k(d(p_i,p_j)/\epsilon)
\end{equation*} 
as the jumping rate from $p_i$ to $p_j$. Here $d$ is the Riemannian metric on $M$. Note that the only dependence on $N$ is through $\epsilon$, hence the notation $W^\epsilon_{ij}$ instead of $W^N_{ij}$. These jumping rates define a random walk on $V_N$. If we regard to points $p_i,p_j$ as having an edge between them if $W^N_{ij}>0$, we want the resulting graph to be connected (to make sense of the random walk and later of the particle systems defined on it). If we assume that there is some $\alpha$ such that $k(x)>0$ for $x\leq \alpha$, one can show that the resulting graph is connected for $N$ large enough. The main reason is that the distance between points that are close to each other goes to zero faster than $\epsilon$. The details of the proof are %in lemma~\ref{eventuallyconnected}
in the appendix. Finally we define
\begin{equation*}
    a(N)=\epsilon^{-2-d}N^{-1}.
\end{equation*}
To prove that the grids are uniformly approximating we have to show~(\ref{unifapproxgrid}), i.e. as the number of points $N$ goes to infinity (and hence the bandwidth $\epsilon$ goes to $0$)
\begin{equation*}
    \frac{1}{N}\sum_{i=1}^N \left|a(N)\sum_{j=1}^NW^\epsilon_{ij}(f(p_j)-f(p_i))- C\Delta_Mf(p_i)\right|\longrightarrow 0 \quad(N\rightarrow\infty).
\end{equation*}
We will prove the following slightly stronger result:
\begin{equation}\label{toprove_grid}
    \sup_{1\leq i\leq N} \left|a(N)\sum_{j=1}^NW^\epsilon_{ij}(f(p_j)-f(p_i))- C\Delta_Mf(p_i)\right|\longrightarrow 0 \quad(N\rightarrow\infty).
\end{equation}
Note that since the process defined above is just a continuous-time random walk its generator is given by
\begin{equation}\label{generatorrandomwalk}
    L^Nf(p_i)=\sum_{j=1}^NW^\epsilon_{ij}(f(p_j)-f(p_i)).
\end{equation}
Therefore we call~(\ref{toprove_grid}) ``convergence of the (rescaled) generators to $\Delta_M$ uniformly in the $p_i$'s for $i\leq N$" or just ``convergence of the generators to $\Delta_M$ uniformly for $i\leq N$". In fact, we will show that the rate of convergence does not depend on $p_i$, so we might as well call it ``uniformly in the $p_i$'s".
\begin{rmk}\label{sgconv}
In fact, we can say more. We denote the semigroups corresponding to the generators $a(N)\sum_{j=1}^NW^\epsilon_{ij}(f(p_j)-f(p_i))$ by $S_t^N$ and the semigroup corresponding to $C\Delta_M$ by $S_t$. Then~(\ref{toprove_grid}) implies that uniformly on compact time intervals
\begin{equation*}
    \sup_{1\leq i\leq N} \left|S_t^Nf|_{G^N}(p_i)- S_tf(p_i)\right|\longrightarrow 0 \quad(N\rightarrow\infty).
\end{equation*}
The proof is a straightforward application of~\cite[Theorem 2.1]{kurtz1969extensions} and a small argument that the extended limit of the generators above (as described in~\cite{kurtz1969extensions}) equals $C\Delta$ since they are equal on the smooth functions.
\end{rmk}
\begin{rmk}
To see why the rescaling $a(N)$ is natural, we can write
\begin{equation*}
    a(N)L^Nf(p_i)=\frac{1}{\epsilon^2}\sum_{j=1}^N\frac{k\left(\frac{d(p_i,p_j)}{\epsilon}\right)}{N\epsilon^d}(f(p_j)-f(p_i)).
\end{equation*}
Since $k$ is a kernel that is rescaled by $\epsilon$ inside, we need the $1/\epsilon^d$ to make sure the integral of the kernel stays of order $1$ as $\epsilon$ goes to $0$. Since the amount of points that the process can jump to equals $N$, we also need the factor $1/N$ to make sure the jumping rate is of order $1$ as $N$ goes to infinity. Also note that the typical distance that a particle jumps with these rates is of order $\epsilon$. This means that space is scaled by $\epsilon$. Hence it is very natural to expect that time should be rescaled by $1/\epsilon^2$, which is exactly what we have. \\
Finally note that in the calculations $N$ is the main parameter and $\epsilon$ an auxiliary parameter depending on $N$. However, conceptually, when the scaling is concerned, the most important parameter is $\epsilon$. $N$ is just the total amount of positions and simply has to grow fast enough as $\epsilon$ goes to $0$. To see why this is true, note that any sequence $\epsilon(N)$ that goes to $0$ more slowly than what we use here will also do. Hence $\epsilon$ should go to $0$ slow enough with respect to $N$ or, equivalently, $N$ should go to infinity fast enough with respect to $\epsilon$.
\end{rmk}
\begin{rmk}
It is also possible to define $W_{ij}^N$ as $p_\epsilon(p_i,p_j)$, the heat kernel after time $\epsilon$, and rescale by $\epsilon^{-1}$ instead of $\epsilon^{-2-d}$. Then the result of section~\ref{replacemeasure} can be proven in the same way (by obtaining some good bounds on Lipschitz constants and suprema of the heat kernel and choosing $\epsilon=\epsilon(N)$ appropriately, see~\cite{gffpaper2018}) and the result of section~\ref{convergenceresult} is a direct consequence of the fact that the Laplace-Beltrami operator generates the heat semigroup. However, for purposes of application/simulation the weights that we have chosen here are much easier to calculate (since only the geodesic distances need to be known, not the heat kernel).
\end{rmk}
\subsection{Replacing empirical measure by uniform measure}\label{replacemeasure}

We would like to show that in this case there is a $C$ independent of $i$ such that for all smooth $f$
\begin{equation*}
    \lim_{N\rightarrow\infty}\epsilon^{-2-d}N^{-1}\sum_{j=1}^N k(d(p_j,p_i)/\epsilon)\left[f(p_j)-f(p_i)\right]= C\Delta_Mf(p_i)
\end{equation*}
uniformly in the $p_i$'s.\\
We can write
\begin{equation}\label{eq1}
    \epsilon^{-2-d}N^{-1}\sum_{j=1}^N k(d(p_j,p_i)/\epsilon)\left[f(p_j)-f(p_i)\right] = \epsilon^{-2-d} \int_M g^{\epsilon,i}\dd\mu^N,
\end{equation}
where 
\begin{equation*}
    g^{\epsilon,i}(p)=k(d(p,p_i)/\epsilon)\left[f(p)-f(p_i)\right].
\end{equation*}
Now~(\ref{eq1}) equals
\begin{equation}\label{twoterms}
    \epsilon^{-2-d} \int_M g^{\epsilon,i}\dd\bar V+\epsilon^{-2-d} \int_M g^{\epsilon,i}\dd(\mu^N-\bar V).
\end{equation}
We will show later that the first term converges to $C\Delta_Mf(p_i)$ (uniformly in the $p_i$'s) as $N\rightarrow\infty$. Therefore it suffices for now to show that the second term converges to $0$, uniformly in the $p_i$'s.\\
Note that $k$ is Lipschitz so it has some Lipschitz constant $L_k<\infty$. This implies that
\begin{equation*}
    \left| k\left(\frac{d(q^1,p_i)}{\epsilon}\right)-k\left(\frac{d(q^2,p_i)}{\epsilon}\right)\right| \leq L_k \left| \frac{d(q^1,p_i)}{\epsilon}-\frac{d(q^2,p_i)}{\epsilon}\right| \leq \frac{L_k}{\epsilon} d(q^1,q^2),
\end{equation*}
by the reverse triangle inequality, so $k(d(\cdot,p_i)/\epsilon)$ has Lipschitz constant $\frac{L_k}{\epsilon}$. $f$ is smooth, so it is Lipschitz too with Lipschitz constant $L_f$. Since $f(p_i)$ is just a constant, $f(\cdot)-f(p_i)$ is also Lipschitz with Lipschitz constant $L_f$. Since they are both bounded functions, we see for the Lipschitz constant of ${g^{\epsilon,j}}$:
\begin{equation*}
    L_{g^{\epsilon,j}}\leq L_{k(d(\cdot,p_i)/\epsilon)} ||f(\cdot)-f(p_i)||_\infty + ||k(d(\cdot,p_i)/\epsilon)||_\infty L_{f(\cdot)-f(p_i)} \leq \frac{2L_k}{\epsilon}||f||_\infty + ||k||_\infty L_f.
\end{equation*}
Note that $k$ is bounded since it is Lipschitz and compactly supported, so $||k||_\infty<\infty$.
This shows that:
\begin{eqnarray*}
    \left|\epsilon^{-2-d} \int_M g^{\epsilon,i}\dd(\mu^N-\bar V)\right|&\leq& \epsilon^{-2-d}\left(\frac{2L_k}{\epsilon}||f||_\infty + ||k||_\infty L_f\right) W_1(\mu^N,\nu)\\
    &=& \epsilon(N)^{-3-d}\left(2L_k||f||_\infty + \epsilon(N)||k||_\infty L_f\right) W_1(\mu^N,\nu),
\end{eqnarray*}
where we denoted the dependence of $\epsilon$ on $N$ explicitly. By~(\ref{defepsilon}), $W_1(\mu^N,\nu)\leq \epsilon(N)^{4+d}$, so we obtain
\begin{equation*}
    \left|\epsilon^{-2-d} \int_M g^{\epsilon,i}\dd(\mu^N-\bar V)\right|\leq \epsilon \left(2L_k||f||_\infty + \epsilon ||k||_\infty L_f\right).
\end{equation*}
Note that this bound does not depend on $p_i$. Since $\epsilon\rightarrow 0$, it follows that the second term of~(\ref{twoterms}) goes to $0$ uniformly in the $p_i$'s.\\
\\
\textbf{What remains}\\
What we have seen above basically means that we can replace the empirical distribution $\mu^N$ by the uniform distribution $\bar V$. For convergence of the generators we still have to show that 
\begin{equation*}
    \lim_{\epsilon\downarrow 0} \epsilon^{-2-d}\int_M k(d(p,p_i)/\epsilon) \left[f(p)-f(p_i)\right] \bar V (\dd p) = C\Delta_Mf(p_i)
\end{equation*}
uniformly in the $p_i$'s. Note that we can replace $N\rightarrow \infty$ by $\epsilon\downarrow 0$, since the expression only depends on $N$ via $\epsilon$ and $\epsilon(N)\downarrow 0$ as $N\rightarrow\infty$. Since the $p_i$'s are all in $M$ we can replace $p_i$ by $q$ and require that the convergence is uniform in $q\in M$.\\
Because of these considerations it remains to show that there exists $C>0$ such that uniformly in $q\in M$:
\begin{eqnarray}\label{whatremains}
    \lim_{\epsilon\downarrow0} \epsilon^{-2-d} \int_M k(d(p,q)/\epsilon) \left[f(p)-f(q)\right] \bar V (\dd p) = C\Delta_Mf(q).
\end{eqnarray}
Note that for every $\epsilon>0$ this expression can be interpreted as the generator of a jump process on the manifold $M$. The process jumps from $p$ to a (measurable) set $Q\subset M$ with rate $\int_Q \epsilon^{-2-d}k(d(p,q)/\epsilon)\dd\bar V$.

\begin{rmk}
Note that this is easy to show in $\R^d$. Indeed, using the transformation $u=(y-x)/\epsilon$ and Taylor, we see
\begin{eqnarray*}
    &&\epsilon^{-2-d} \int_{\R^d} k\left(\frac{\|y-x\|}{\epsilon}\right) (f(y)-f(x))\dd y = \epsilon^{-2} \int_{\R^d} k(\|u\|) (f(x+\epsilon u)-f(x))\dd u\\
    &=& \epsilon^{-1} \int_{\R^d} k(\|u\|) \grad f(x) \cdot u \dd u + \frac{1}{2} \int_{\R^d} k(\|u\|) u^TH(x)u\dd u + O(\epsilon),
\end{eqnarray*}
where $H(x)$ is the Hessian of $f$ in $x$. Now changing coordinates to integrate over each sphere $B_r$ of radius $r$ with respect to the appropriate surface measure $S_r$ and then with respect to $r$, we obtain
\begin{equation*}
    \epsilon^{-1} \int_{\R} k(r) \int_{B_r} \grad f(x) \cdot w S_r(\dd w)\dd r + \frac{1}{2} \int_{\R} k(r) \int_{B_r} w^TH(x)w S_r(\dd w)\dd r + O(\epsilon).
\end{equation*}
Now because of symmetry the integrals of $w_i$ and of $w_iw_j$ over spheres vanish for each $i\neq j$. Moreover the integrals of $w_i^2$ do not depend on $i$, but only on $r$. Therefore the first term vanishes and we are left with
\begin{equation*}
    \frac{1}{2} \int_{\R} k(r) C(r) \Delta f(x) \dd r + O(\epsilon) = C'\Delta f(x) + O(\epsilon).
\end{equation*}
This shows convergence (at least pointwise, for uniform convergence we have to be a little more careful about the $O(\epsilon)$).
\end{rmk}
\subsection{Convergence result}\label{convergenceresult}

\textbf{Integral over tangent space}\\
Let $\alpha>0$ be such that $\supp~k\subset [0,\alpha]$ (such $\alpha$ exists since $k$ is compactly supported). We denote for $p\in M,r>0: B_d(p,r)=\{q\in M: d(p,q)\leq r\}$. Then we can write
\begin{equation}\label{kfilledin}
    \int_M k(d(p,q)/\epsilon)(f(q)-f(p)) \bar V(\dd q) = \int_{B_d(p,\alpha\epsilon)} k(d(p,q)/\epsilon) (f(q)-f(p)) \bar V(\dd q).
\end{equation}\sk

Denote for $\eta\in T_pM, r>0: B_p(\eta,r)=\{\xi\in T_pM: ||\xi-\eta||\leq r\}$ (not to be confused with $B_\rho$, which is a ball in $M$ with respect to the original metric $\rho$). For $\epsilon$ small enough we know that $\exp_p: T_pM\supset B_p(0,\alpha\epsilon) \rightarrow B_d(p,\alpha\epsilon)\subset M$ is a diffeomorphism. We want to use this to write the integral above as an integral over $B_p(0,\epsilon)\subset T_pM$:
\begin{eqnarray}
    \int_{B_d(p,\alpha\epsilon)} k(d(p,q)/\epsilon) (f(q)-f(p)) \bar V(\dd q) &=& \int_{B_p(0,\alpha\epsilon)} k(d(p,\exp_p(\eta))/\epsilon) (f(\exp_p(\eta) )-f(p)) \bar V\circ\exp(\dd \eta)\nonumber \\
    &=& \int_{B_p(0,\alpha)} k(d(p,\exp_p(\epsilon \eta))/\epsilon) (f(\exp_p(\epsilon\eta))-f(p)) \bar V\circ\exp\circ\lambda_\epsilon(\dd \eta) \nonumber \\ 
    &=& \int_{B_p(0,\alpha)} k(||\eta||) (f(\exp_p(\epsilon\eta))-f(p)) \bar V\circ\exp\circ\lambda_\epsilon(\dd \eta).\label{integral}
\end{eqnarray}
This means we integrate with respect to the measure $\bar V\circ\exp\circ\lambda_\epsilon$, where $\lambda_\epsilon$ denotes multiplication with $\epsilon$.\\
\\
\textbf{Determining the measure $\bar V\circ\exp\circ\lambda_\epsilon$}\\
Since $B_p(0,\alpha\epsilon)$ is a star-shaped open neighbourhood of $0$, we see that for $\epsilon$ small enough $V_\epsilon:=B_d(p,\alpha\epsilon)=\exp_p(B_p(0,\alpha\epsilon))$ is a normal neighbourhood of $p$, so there exists a normal coordinate system $(x,V_\epsilon)$ that is centered at $p$. We interpret, for $v\in \R^n$, $v_p\in T_pM$ as $\sum_i v_i \frac{\partial}{\partial x^i}$. Consequently, when we write $A_p$ for some subset $A$ of $\R^n$, we mean $\{v_p: v\in A\}$. Since the basis $W=\left(\frac{\partial}{\partial x^1}...,\frac{\partial}{\partial x^n}\right)$ is orthogonal in $T_pM$, it is easy to see that $\phi:=v_p\mapsto v$ preserves the inner product and is an isomorphism of inner product spaces. Indeed, 
\begin{equation*}
    ||v_p||^2=\left<v_p,v_p\right>=(v_p)^i(v_p)^jg_{ij}=\sum_{ij}v^iv^j\delta^i_j=\sum_i (v^i)^2=||v||^2.
\end{equation*} 
In particular $B_{\R^n}(0,\alpha\epsilon)_p=B_p(0,\alpha\epsilon)$ (where $B_{\R^n}$ denotes a ball in $\R^n$ with respect to the Euclidean metric). We can use this in the following lemma, which tells us more about $\bar V\circ\exp\circ\lambda_\epsilon$.
\begin{lemma}\label{measure} There exist $\epsilon'>0$ and a function $h:B_{\R^n}(0,\epsilon')\rightarrow \R$ such that for $t$ tending to $0$ $h(t)=O(||t||^2)$ and for all $0<\epsilon<\epsilon'$:
$\bar V\circ \exp\circ\lambda_\epsilon=\epsilon^n \left(\frac{1+h(\epsilon t)}{V(M)} \dd t^1..\dd t^n \right)\circ\phi$ on $B_p(0,\alpha)$.
\end{lemma}
\begin{proof}
Let $\epsilon'$ be small enough such that the considerations above the lemma hold and let $\epsilon<\epsilon'$. For clarity of the proof, we first separately prove the following statement.\\
\\
\textit{\underline{Claim:}} $x\circ\exp=\phi$ on $B_{\R^n}(0,\alpha\epsilon)_p$.\\
\textit{Proof.} The geodesics through $p$ are straight lines with respect to $x$, so they are of the form $x(\gamma(t))=ta+b$ with $a,b\in\R^n$. For $\eta=\sum_i \eta^i\frac{\partial}{\partial x^i}$, the geodesic starting at $p$ with tangent vector $\eta$ at $p$ should satisfy $b=x(p)=0$ and $a_i=\eta^i$ for all $i$, so we see $\gamma^k=t\eta^k$.  For $q\in B_d(p,\alpha\epsilon)$, we see $x^k(\exp(x(q)_p))=1*x^k(q)=x^k(q)$, so $\exp(x(q)_p)=q$. This also shows that $x\circ\exp(v_p)=v$ for $v\in B_{\R^n}(0,\alpha\epsilon)$ (since $x$ is invertible), which gives an identification 
\begin{equation*}
    x\circ \exp: T_pM\supset B_{\R^n}(0,\alpha\epsilon)_p \rightarrow B_{\R^n}(0,\alpha\epsilon)\subset\R^n
\end{equation*} 
which is the restriction of $\phi$ to $B_{\R^n}(0,\alpha\epsilon)_p$. This situation is sketched in figure~\ref{proof_measure}. \qed\\
\\
Now we will first use the definition of integration to see what the measure is in coordinates (so it becomes a measure on a subset of $\R^n$). Then we will use the claim above: we will pull the measure on $\R^n$ back to $T_pM$ using $\phi$.\\
On $(x,V_\epsilon)$ the volume measure is given by $\sqrt{\det G} \dd x^1\wedge..\wedge\dd x^n$. According to~\cite[Cor 2.3]{wangzuoq2016normal}, $\sqrt{\det G}$ can be expanded (in normal coordinates) as $1+h(x)$ where $h$ is such that $h(x)=O(||x||^2)$. Now the measure can be written in local coordinates on $B_{\R^n}(\alpha\epsilon')$ as $(1+h(x))\dd x^1\wedge..\wedge\dd x^n$, so the uniform measure is $\frac{1+h(x)}{V(M)}\dd x^1\wedge..\wedge\dd x^n$. This yields the measure $\bar V\circ x^{-1}=\frac{1+h(t)}{V(M)} \dd t^1..\dd t^n$ on $x(V_{\epsilon'})=B_{\R^n}(0,\alpha\epsilon')$. We have on $B_{\R^n}(0,\alpha)_p$:
\begin{equation*}
    \bar V\circ\exp\circ\lambda_\epsilon=(\bar V\circ x^{-1})\circ(x\circ\exp)\circ\lambda_\epsilon.
\end{equation*}
According to the claim above, $x\circ\exp$ is a restriction of $\phi$, so we can replace it by $\phi$. Since this map is linear, it can be interchanged with $\lambda_\epsilon$, which yields (inserting what we found before and since $\epsilon<\epsilon'$):
\begin{equation*}
    \left(\frac{1+h(t)}{V(M)} \dd t^1..\dd t^n \right)\circ\lambda_\epsilon\circ\phi=\left(\frac{\epsilon^n(1+h(\epsilon t))}{V(M)} \dd t^1..\dd t^n \right)\circ\phi.
\end{equation*}
In the last step we interpret $\frac{\epsilon^n(1+h(\epsilon t))}{V(M)} \dd t^1..\dd t^n$ as a measure on $B_{\R^n}(0,\alpha)$ and this last step is then just a transformation of measures on $\R^n$. This yields the expression that we want.
\end{proof}

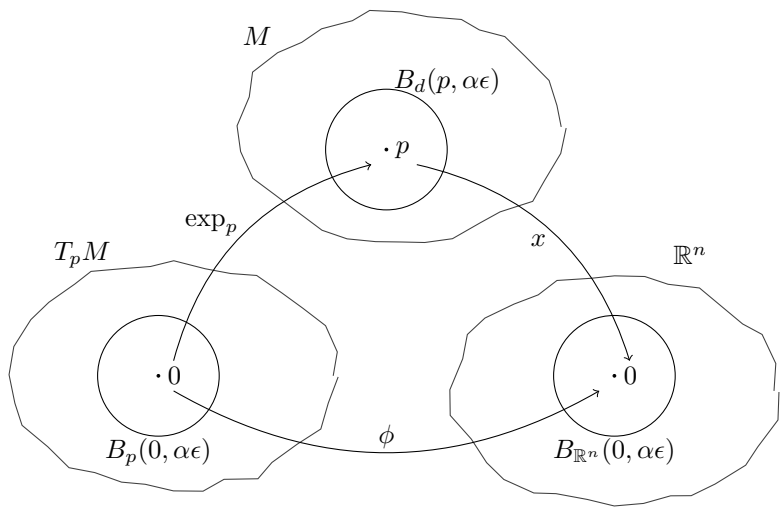
\begin{figure}[ht]
\begin{centering}
\begin{tikzpicture}[pencildraw/.style={
    black!75,
    decorate,
    decoration={random steps,segment length=10pt,amplitude=2pt}
    }
]
\draw (3,2) circle [radius=0.8];
\draw (9,2) circle [radius=0.8];
\draw (6,5) circle [radius=0.8];

\node [right] at (6,5) {$p$};
\draw [fill] (6,5) circle [radius=0.5pt];
\node [right] at (9,2) {$0$};
\draw [fill] (9,2) circle [radius=0.5pt];
\node [right] at (3,2) {$0$};
\draw [fill] (3,2) circle [radius=0.5pt];

\draw [pencildraw] (3.2,2) ellipse (2.1cm and 1.5cm);
\draw [pencildraw] (9,1.8) ellipse (2.1cm and 1.5cm);
\draw [pencildraw] (6.2,5.3) ellipse (2.1cm and 1.5cm);

\draw [->] (3.2,2.2) to [bend left] (5.8,4.8);
\draw [->] (6.4,4.8) to [bend left] (9.2,2.2);
\draw [->] (3.2,1.8) to [bend right] (8.8,1.8);

\node at (2,3.6) {$T_pM$};
\node at (4.3,6.5) {$M$};
\node at (10,3.6) {$\R^n$};

\node at (3.7,4) {$\exp_p$};
\node at (8,3.8) {$x$};
\node at (6,1.2) {$\phi$};

\node at (6.8,5.9) {$B_d(p,\alpha\epsilon)$};
\node [below] at (3,1.3) {$B_p(0,\alpha\epsilon)$};
\node [below] at (9,1.3) {$B_{\R^n}(0,\alpha\epsilon)$};
\end{tikzpicture}
\caption{The situation in lemma~\ref{measure}. On $B_p(0,\alpha\epsilon)$: $x\circ\exp=\phi$. The uniform measure on $B_d(p,\alpha\epsilon)$ is moved via $x$ to $B_{\R^n}(0,\alpha\epsilon)$ using the formula $\sqrt{\det G}t_1..t_n$. This measure can then be pulled back to $B_p(0,\alpha\epsilon)$ using $\phi$. Since $\phi$ is an inner product space isomorphism, it will be easy to deal with orthogonal transformations later, in lemma~\ref{mucan}.}
\label{proof_measure}
\end{centering}
\end{figure}

\begin{rmk}
We used~\cite[Cor 2.3]{wangzuoq2016normal} in the proof above. In these notes the expansion of $\sqrt{\det G(p,x)}$ is calculated around a point $p$ in normal coordinates $x$ centered around $p$:
\begin{equation}\label{toProveUniform}
    \sqrt{\det G(p,x)} = 1 - \frac{1}{6}\text{Ric}(p)_{kl}x^kx^l+O\left(|x|^3\right).
\end{equation}
As can be seen, there are no linear terms in the expansion. The coefficients for the quadratic terms are coefficients of the Ricci curvature of $M$ in $p$.  This implies that the way that the uniform distribution on a ball around $p$ in $M$ is pulled back to the tangent space via the exponential map depends on the curvature of $M$ in $p$. In particular, if there is no curvature, $M$ is locally isomorphic to a neighbourhood in $\R^n$ so the same thing happens as in $\R^n$. This means that we get a uniform distribution on a ball around $0$ in the tangent space.
\end{rmk}\sk
\begin{rmk}\label{gap}
We will need in proposition~\ref{endprop} that the statement of lemma~\ref{measure} holds uniformly in all points of the manifold. This means that the difference between the uniform measure on a ball in the tangent space and the pulled back uniform measure on a geodesic ball in the manifold decays quadratically with $\epsilon$ uniformly in the manifold. Note that this uniform convergence is intuitively clear, since the difference between the two measures is caused by curvature and curvature is bounded in a compact manifold. As in the proof of lemma~\ref{measure}, one needs to write
\begin{equation*}
    \sqrt{\det G}(\exp_p(x)) = 1 + h_p(x)
\end{equation*}
for some function $h_p$ that is $O(|x|^2)$ independent of $p$. Here $G(q)$ is the metric matrix at $q$ expressed in (fixed) normal coordinates centered at $p$. Since $\sqrt{\hphantom{a}}$ and $\det$ are uniformly continuous in the right domains, it suffices to show that 
\begin{equation}\label{quadapprox}
    G(\exp_p(x)) = I + O(|x|^2),
\end{equation}
where the $O(|x|^2)$ is independent of $p$. In other words,
\begin{equation}\label{quadapproxprecise}
    ||G(\exp_p(x))-I||\leq C||x||^2,
\end{equation}
where $C$ does not depend on $p$.
For all $p\in M$ (and for any system of normal coordinates centered at $p$) we have the following Taylor expansion (note that for fixed $p$ $G(\exp_p(\cdot))_{ij}$ is a map from a (subset of) $\R^d$ to $\R$):
\begin{equation}\label{Taylor}
    G(\exp_p(x))_{ij} = \delta_{ij} + \frac{1}{3}R_{ijkl}x^kx^l + \sum_{|\beta|=3} \frac{3}{\beta!}\int_0^1 (1-t)^2 D^\beta G(\exp_p(\cdot))_{ij}(tx)\dd t\cdot x^\beta.
\end{equation}
From this we get~\eqref{quadapproxprecise} directly for fixed $p$, i.e. we have
\begin{equation*}
    ||G(\exp_p(x))-I||\leq C_p||x||^2.
\end{equation*} 
In order to obtain uniformity of $C_p$ in $p$, we note that the functions of $p$ and $x$ appearing in the r.h.s. of~\eqref{Taylor} can be made smooth both in $p$ and $x$. Smoothness in $x$ is obvious (within the injectivity radius) and smoothness in $p$ follows from a special choice of normal coordinates in such a way that they vary smoothly with $p$. A choice of normal coordinates is equivalent to a choice of an orthonormal basis, so one can construct smoothly varying normal coordinates by taking a smooth section of the orthonormal frame bundle (this can only be done locally, but it is enough to have the uniformity result locally, since then by compactness one has it globally). By compactness, the injectivity radius is bounded from below by some $\delta>0$. Now for all $p\in M$ and $||x||<\delta$,~\eqref{Taylor} holds and (locally) the quantities on the r.h.s. vary smoothly and therefore (again by compactness) one can show that $C:=\sup_p C_p$ is finite.
\end{rmk}

\textbf{A canonical part plus a rest term}\\
Now define
\begin{equation*}
    \mu=\left(\frac{1}{V(M)} \dd t^1..\dd t^n \right)\circ\phi \hspace{0.6cm} \text{  and  } \hspace{0.6cm}
    \mu_R=\left(\frac{h(\epsilon t)}{V(M)} \dd t^1..\dd t^n \right)\circ\phi
\end{equation*}
on $B_p(0,\alpha)$ and $0$ everywhere else. 
Then the lemma implies that~(\ref{integral}) equals
\begin{eqnarray*}
    &&\int_{B_p(0,\alpha)} k(||\eta||) (f(\exp_p(\epsilon\eta))-f(p)) \epsilon^n (\mu+\mu_R)(\dd \eta) =\epsilon^n  \int_{T_pM}  (f(p(\epsilon,\eta))-f(p)) k(||\eta||) (\mu+\mu_R)(\dd \eta).
\end{eqnarray*}
Recall that $p(\epsilon,\eta)$ is just notation for following the geodesic from $p$ in the direction of $\eta$ for time $\epsilon$. Now we define $\mu^k=k(||\cdot||)\mu$ (so the measure which has density $k(||\cdot||)$ with respect to $\mu$) and analogously $\mu^k_R=k(||\cdot||)\mu_R$. Then we can write the integral above as
\begin{equation*}
    \epsilon^n  \int_{T_pM}  (f(p(\epsilon,\eta))-f(p)) (\mu^k+\mu_R^k)(\dd \eta).
\end{equation*}
In this way we transformed the integral to one that we worked with in section~\ref{1generators} since we wrote it as the generator of a geodesic random walk (see $L_N$ on page~\pageref{1generators}). To use the theory that we obtained in that section, we need the following lemma. It tells us that $\mu^k$ can be used as a stepping distribution for a geodesic random walk and it gives us the constant speed of the Brownian motion to which it converges (see section~\ref{3stepdist}).
\begin{lemma}\label{mucan}
$\mu^k$ is canonical. Moreover $\int_{T_pM} ||\eta||^2 \mu^k(\dd\eta)=\frac{2\pi^{n/2}}{V(M)\Gamma(n/2)}\int_0^\infty k(r)r^{n+1}\dd r$.
\end{lemma}
\begin{proof}
First of all recall that $k$ is continuous and compactly supported, so the integral over $k$ above makes sense and is finite. Define $\nu=\frac{1}{V(M)} \dd t^1..\dd t^n$ on $B_{\R^n}(0,\alpha)$ and $0$ everywhere else. Then we can write $\mu=\nu\circ\phi$. Since $\phi$ preserves the norm, we see that $k(||\cdot||_{T_pM})\circ \phi^{-1}=k(||\cdot||_{\R^n})$. This means that $\mu^k=\nu^k\circ\phi$, where $\nu^k:=k(||\cdot||)\nu$. Since $\phi$ preserves the inner product, the measure $\mu^k$ behaves the same with respect to orthogonal transformations in $T_pM$ as $\nu^k$ with respect to orthogonal transformations in $\R^n$. Since $\nu^k$ is clearly preserved under such transformations, so is $\mu^k$. This shows that $\mu^k$ is canonical.\\
Now we calculate the corresponding constant.
\begin{eqnarray*}
\int_{T_pM} ||\eta||_{T_pM}^2\mu^k(\dd \eta) &=& \int_{T_pM} ||v_p||_{T_pM}^2\mu^k(\dd v_p) = \int_{\R^n} ||\phi^{-1}(v)||_{T_pM}^2\nu^k(\dd v)
\\ &=& \int_{\R^n} ||v||_{\R^n}^2\nu^k(\dd v) = \frac{1}{V(M)} \int_{B_{\R^n}(0,\alpha)} ||v||_{\R^n}^2k(||v||_{\R^n})\dd v
 \\ &=& \frac{1}{V(M)}\int_0^\alpha r^2 k(r) \frac{2\pi^{n/2}}{\Gamma(n/2)}r^{n-1}\dd r = \frac{2\pi^{n/2}}{V(M)\Gamma(n/2)}\int_0^\infty k(r)r^{n+1}\dd r
\end{eqnarray*}
The first step was just writing the integral with respect to the coordinates for which we defined $\mu$. The second step holds because $\mu^k=\nu^k\circ\phi$. The third uses the fact that $\phi$ preserves the norm. The penultimate step is a change of coordinates in $\R^n$ using the fact that $||v||$ is constant on spheres around the origin. Here $\frac{2\pi^{n/2}}{\Gamma(n/2)}r^{n-1}$ is the area of $rS_{n-1}$. In the last step we used that $\supp (k)\subset[0,\alpha]$.
\end{proof}\sk

\textbf{Conclusion}\\
We use everything above to obtain the statement that we aim for.
\begin{prop}\label{endprop}
Set
\begin{equation*}
    C=\frac{\pi^{n/2}}{V(M)n\Gamma(n/2)}\int_0^\infty k(r)r^{n+1}\dd r.
\end{equation*}
Then as $\epsilon\rightarrow 0$ we have uniformly in $p\in M$:
\begin{eqnarray*}
    \epsilon^{-2-n} \int_M k(d(p,q)/\epsilon) \left[f(q)-f(p)\right] \bar V (\dd q) \longrightarrow C \Delta_M f(p)
\end{eqnarray*}
\end{prop}
\begin{proof}
Let $p\in M$. We can write
\begin{eqnarray*}
    &&\int_M k(d(p,q)/\epsilon)(f(q)-f(p)) \bar V(\dd q)
    =  \epsilon^n \int_{T_pM} (f(p(\epsilon,\eta))-f(p)) (\mu^k+\mu^k_R)(\dd \eta)\\
    &=&\epsilon^n \int_{T_pM} (f(p(\epsilon,\eta))-f(p)) \mu^k(\dd \eta)+ \epsilon^n \int_{T_pM} (f(p(\epsilon,\eta))-f(p))^2 \mu_R^k(\dd \eta).
\end{eqnarray*}
From the results in section~\ref{1generators} and~\ref{3stepdist} (prop~\ref{sumprop}) and lemma~\ref{mucan}, we see for the first term uniformly in $p$
\begin{eqnarray*}
    \lim_{\epsilon \downarrow 0} \frac{1}{\epsilon^{2+n}} \epsilon^n \int_{T_pM} (f(p(\epsilon,\eta))-f(p)) \mu^k(\dd \eta) &=& \lim_{\epsilon \downarrow 0}\frac{1}{\epsilon^{2}} \int_{T_pM} (f(p(\epsilon,\eta))-f(p)) \mu^k(\dd \eta)\\
    &=&\frac{1}{n} \frac{2\pi^{n/2}}{V(M)\Gamma(n/2)}\int_0^\infty k(r)r^{n+1}\dd r\cdot \frac{1}{2} \Delta_M f(p)=C \Delta_M f(p).
\end{eqnarray*}
Now it suffices to show that the second term goes to zero at a rate independent of $p$. Let $\epsilon'',K>0$ such that $\epsilon''<\epsilon'$ and $|h(s)|<K||s||^2$ for $s\in B_{\R^n}(0,\epsilon'')$ (where both $\epsilon'$ and $h$ are from lemma~\ref{measure}). We need remark~\ref{gap} to make sure that $K$ and $\epsilon''$ do not depend on $p$. Now note that for $\epsilon<\epsilon''$:
\begin{equation*}
    |\mu_R|\leq \left(\sup_{t\in B_{\R^n}(0,1)}|h(\epsilon t)|\right)\mu \leq \left(\sup_{t\in B_{\R^n}(0,1)} K||\epsilon t||^2\right)\mu = \left(\sup_{t\in B_{\R^n}(0,1)} K\epsilon^2 ||t||^2\right)\mu = K\epsilon^2\mu.
\end{equation*}
Now we see:
\begin{eqnarray*}
    &&\lim_{\epsilon \downarrow 0} \frac{1}{\epsilon^{2+n}} \epsilon^n \left|\int_{T_pM} f(p(\epsilon,\eta))-f(p) \mu^k_R(\dd \eta)\right| \leq \lim_{\epsilon \downarrow 0} \frac{1}{\epsilon^{2}} \int_{T_pM} \left|f(p(\epsilon,\eta))-f(p)\right| k(||\eta||) |\mu_R|(\dd \eta)\\
    &\leq& \lim_{\epsilon \downarrow 0} \frac{1}{\epsilon^{2}} \int_{T_pM} d(p(\epsilon,\eta),p)L_f k(||\eta||) K\epsilon^2 \mu(\dd \eta)
    \leq L_f K \lim_{\epsilon \downarrow 0} \int_{T_pM} \epsilon ||\eta|| k(||\eta||) \mu(\dd \eta) \\
    &=& L_f K  \int_{T_pM} ||\eta|| k(||\eta||) \mu(\dd \eta)\lim_{\epsilon \downarrow 0}\epsilon = 0,
\end{eqnarray*}
where we used that the integral is finite since $k$ is bounded and has support in $[0,\alpha]$. Combining everything above gives what we wanted.
\end{proof}
\subsection{Example grid}\label{randomGrid}
So far, we have seen that a sequence of grids is suitable for the hydrodynamic limit problem if the empirical distributions converge to the uniform distribution in the Kantorovich topology. We conclude by giving examples of such grids. To be more precise, we show that if one constructs a grid by adding uniformly sampled points from the manifold, this grid is suitable with probability 1.\\
\\
\begin{rmk}[Comparison with standard grids] Recall the grids $S^N$ on the one-dimensional torus $S$ from remark~\ref{comprmk}. We can show that the empirical measures corresponding to these grids along the subsequence $N=2^m, m=0,1,2,..$ converge to the uniform measure on $S$ with respect to the Kantorovich distance. To this end let $N=2^m$ be fixed, call the corresponding empirical measure $\mu^N$ and call the uniform measure $\lambda$. Recall that the Kantorovich distance between these measures is alternatively given by
\begin{equation*}
    W_1(\mu^N,\lambda)=\inf_{\gamma\in \Gamma(\mu^N,\lambda)} \int_{S\times S} d(x,y) \gamma(dx,dy),
\end{equation*}
where $\Gamma(\mu^N,\lambda)$ is the set of all couplings of $\mu^N$ and $\lambda$. Now let $Y$ be a uniform random variable on $S$ and define
\begin{equation*}
    X=k/N \iff Y\in \left[\frac{k-1/2}{N},\frac{k+1/2}{N}\right).
\end{equation*}
Denote the joint distribution of $(X,Y)$ by $\nu$. Then it is easy to see that $\nu\in\Gamma(\mu^N,\lambda)$. This implies that
\begin{equation*}
    W_1(\mu^N,\lambda)\leq \int_{S\times S} d(x,y) \nu(dx,dy) = \E_\nu (d(X,Y)) \leq \frac{1}{2N}.
\end{equation*}
This implies convergence with respect to the Kantorovich metric along the subsequence $N=2^m, m=0,1,2,..$. Note, however, that the corresponding edge weights as described in this section are \textit{not} the same as those in remark~\ref{comprmk}.
\end{rmk}
\textbf{Convergence of a random grid}\\
Now we move back to the general case of a compact and connected $n$-dimensional Riemannian manifold $M$.
Let $(P_n)_{n=1}^\infty$ be a sequence of iid uniformly random points of $M$. Define $ \mu^N=\frac{1}{N}\sum_{i=1}^N \delta_{P_i}$. We follow~\cite[Example 5.15]{handel2016prob} to show that $W_1(\mu^N,\bar V)\rightarrow 0$ as $N\rightarrow\infty$. 
First we will show that the expectation goes to $0$, then we will derive that it goes to $0$ almost surely.\\
For now, let $N$ be fixed. Let $\mathscr{F}_1$ be the set of Lipschitz function on $M$ with Lipschitz constant $\leq 1$. Then we define for $f\in \mathscr{F}_1$ the random variable $X_f=\mu^Nf-\bar V f$. Note that both $\mu^N$ and $\bar V$ are probability distributions, so $X_f(\omega)$ is Lipschitz in $f$ for each $\omega$:
\begin{equation*}
    |X_f-X_g|=|\mu^Nf-\bar V f-(\mu^Ng-\bar V g)|\leq |\mu^N(f-g)|+|\bar V(f-g)|\leq 2||f-g||_\infty.
\end{equation*}
Now note that since $f$ has Lipschitz constant $\leq 1$:
\begin{equation*}
    \sup_{p\in M}f(p)-\inf_{q\in M}f(q) = \sup_{p,q\in M} |f(p)-f(q)|\leq \sup_{p,q\in M}d(p,q) =: K.
\end{equation*}
$M$ is compact, so $K<\infty$. Since adding constants to $f$ does not change $X_f$, it suffices to consider $f\in \mathscr{F}_{1,K}=\{g\in \mathscr{F}_1: 0\leq g \leq K\}$. It follows that for each $f\in \mathscr{F}_{1,K}$ by writing
\begin{equation*}
    X_f = \sum_{i=1}^N \frac{f(X_i)-\bar V f}{n},
\end{equation*}
we see that it is a sum of iid random variables taking values in $[-\frac{K}{N},\frac{K}{N}]$. By the Azuma-Hoeffding inequality, this implies that $X_f$ is $\frac{K^2}{N}$-subgaussian for each $f\in \mathscr{F}_{1,K}$. Now~\cite[Lemma 5.7]{handel2016prob} shows that
\begin{equation*}
    \E [W_1(\mu^N,\bar V)]\leq \inf_{\epsilon>0}\left\{ 2\epsilon+\sqrt{\frac{2K^2}{N} \log N(W,||\cdot||_\infty,\epsilon)}\right\},
\end{equation*}
where $N(\mathscr{F}_{1,K},||\cdot||_\infty,\epsilon)$ is the minimal number of points in some space containing $\mathscr{F}_{1,K}$ such that the balls of radius $\epsilon$ with respect to the uniform distance around those points cover $\mathscr{F}_{1,K}$. \\
\\
\textbf{Estimating the covering number $N(\mathscr{F}_{1,K},||\cdot||_\infty,\epsilon)$}\\
We now need to estimate this covering number. To do this we need an upper bound of the covering number $N(M,d,\epsilon)$ of $M$. Since $M$ is compact there exist $a,\delta>0$ such that for all $0<\epsilon<\delta$: $N(M,d,\epsilon)\leq a\epsilon^{-d}$ (see for instance~\cite[Lemma 4.2]{loubes2008kernel}). Using this we can prove the following.
\begin{lemma}
There is a $c>0$ such that for all $0<\epsilon<\delta$: $N(\mathscr{F}_{1,K},||\cdot||_\infty,\epsilon)\leq \exp{c/\epsilon^d}.$
\end{lemma}
\begin{proof}
Fix $\epsilon>0$ and call $m=N(M,d,\epsilon/4)$. By definition of this number, we can find points $p_1,..,p_m\in M$ such that $\cup_{i=1}^m B(p_i,\epsilon/4)\supset M$. Now define $V_1=B(p_1,\epsilon/4)$ and for $i\geq 2$: $V_i=B(p_i,\epsilon/4)\setminus \cup_{j=1}^{i-1}V_j$. Now for $f\in \mathscr{F}_{1,K}$, define $\pi^f:M\rightarrow\R$ by 
\begin{equation*}
    \pi^f: V_i\ni p\mapsto \epsilon \left(\left\lfloor \frac{f(p_i)}{\epsilon}\right\rfloor+\frac{1}{2}\right).
\end{equation*}
%\begin{equation*}
%    \pi^f: V_i\ni p\mapsto \frac{\epsilon}{2} \left(\left\lceil \frac{f(p_i)}{\epsilon}\right\rceil %+ \left\lfloor \frac{f(p_i)}{\epsilon}\right\rfloor\right).
%\end{equation*}
Since each $p\in M$ is contained in exactly one $V_i$ (by construction), this map is well-defined. Note that if $k\epsilon\leq f(p_i)<(k+1)\epsilon$, then $\pi^f=(k+1/2)\epsilon$ on $V_i$. In particular clearly $|f(p_i)-\pi^f(p_i)|\leq \epsilon/2$. Now denote $Y=\{\pi^f|f\in \mathscr{F}_{1,K}\}$.\\
Now fix $f\in \mathscr{F}_{1,K}$ and $p\in M$. Let $i$ be such that $p\in V_i$. Then we see:
\begin{equation*}
    |\pi^f(p)-f(p)|= |\pi^f(p_i)-f(p)|\leq |\pi^f(p_i)-f(p_i)|+|f(p_i)-f(p)|\leq \epsilon/2 + L_fd(p_i,p) \leq \epsilon/2+\epsilon/4< \epsilon.
\end{equation*}
This shows that $||\pi^f-f||_\infty\leq \epsilon$, which implies that $Y$ is an $\epsilon$-net for $\mathscr{F}_{1,K}$. Hence $N(\mathscr{F}_{1,K},||\cdot||_\infty,\epsilon)\leq \#Y$.\\
\\
All we have to do now is estimate $\#Y$. \\
First of all let $\pi^f\in Y$. Note that if $d(p_i,p_j)\leq \epsilon/2$, we see
\begin{eqnarray*}
    |\pi^f(p_i)-\pi^f(p_j)|&\leq& |\pi^f(p_i)-f(p_i)|+|f(p_i)-f(p_j)|+|f(p_j)-\pi^f(p_j)|\\
    &\leq& \epsilon/2 + L_fd(p_i,p_j)+\epsilon/2=3\epsilon/2.
\end{eqnarray*}
Since $|\pi^f(p_i)-\pi^f(p_j)|=k\epsilon$ for some $k\in\Z$, we conclude $|\pi^f(p_i)-\pi^f(p_j)|\in\{-\epsilon,0,\epsilon\}$, so $\pi^f(p_i)\in\{\pi^f(p_j)-\epsilon,\pi^f(p_j),\pi^f(p_j)+\epsilon\}.$\\
\\
Now define a graph $G$ with vertices $p_1,..,p_m$ by putting an edge between $p_i$ and $p_j$ whenever $d(p_i,p_j)\leq\epsilon/2$. Any $\pi^f$ is uniquely specified by its values on the nodes of $G$. Note further that whenever we know $\pi^f$ for some point of the graph, there are only 3 possible values left for each of its neighbours (since neighbours are at distance at most $\epsilon/2$). Now $\#Y$ is dominated by the amount of ways in which we can assign values of the type $(k+1/2)\epsilon$ to nodes of $G$ while keeping this restriction into account. Define, for $i\leq 0$, $S_i=\{p\in G: d_G(p_1,p)=i\}$, where $d_G(p,q)$ denotes the minimum amount of edges that need to be followed to walk from $p$ to $q$ in $G$. Now we can start counting.\\
For $p_1$, there are at most $\left\lceil K/\epsilon \right\rceil$ possible values (recall that any $f\in \mathscr{F}_{1,K}$ has $0\leq f\leq K$). Each node in $S_1$ is a distance at most $\epsilon/2$ from $p_1$, so each node can take at most $3$ values. This brings the possible amount of value assignments to (less than) $\left\lceil K/\epsilon \right\rceil 3^{\#S_1}$. Now each node in $S_2$ is at distance at most $\epsilon/2$ of a node in $S_1$, so each of these can take at most 3 different values. This brings the number of options so far to at most $\left\lceil K/\epsilon \right\rceil 3^{\#S_1}3^{\#S_2}$. Continuing in this way, we obtain that the number of ways to assign values is at most
\begin{equation*}
    \left\lceil \frac{K}{\epsilon} \right\rceil \prod_{i=1}^{\infty}3^{\#S_i} = \left\lceil \frac{K}{\epsilon} \right\rceil 3^{\sum_{i=1}^{\infty}\#S_i} = \left\lceil \frac{K}{\epsilon} \right\rceil 3^{m-1} = \left\lceil \frac{K}{\epsilon} \right\rceil 3^{N(M,d,\epsilon/4)-1}.
\end{equation*}
Recall that $m$ is the total amount of balls as we defined at the beginning of the proof, which we chose equal to $N(M,d,\epsilon/4)$.
Now we know that for $0<\epsilon<\delta$
\begin{equation*}
    N(\mathscr{F}_{1,K},||\cdot||_\infty,\epsilon)\leq \left\lceil \frac{K}{\epsilon} \right\rceil 3^{a/(\epsilon/4)^d-1} = \left\lceil \frac{K}{\epsilon} \right\rceil 3^{a4^d/\epsilon^d-1}.
\end{equation*}
This implies that there exists $c>0$ such that for all $0<\epsilon<\delta$  $N(\mathscr{F}_{1,K},||\cdot||_\infty,\epsilon)\leq \e^{c/\epsilon^d}.$
\end{proof}

Now we see that for any $0<\epsilon<\delta:$ 
\begin{equation*}
    \E [W_1(\mu^N,\bar V)]\leq 2\epsilon + \sqrt{\frac{2K^2}{N} \log \exp{c/\epsilon^d}} = 2\epsilon+\sqrt{\frac{2cK^2}{N}}\epsilon^{-d/2}.
\end{equation*}
Elementary methods show that this value takes a minimum at $\epsilon=c_0N^{\frac{-1}{d+2}}$ where $c_0$ is some constant (take $N$ large enough such that $c_0N^{\frac{-1}{d+2}}<\delta$). This shows that the optimal bound that we get is
\begin{equation*}
    2c_0N^{\frac{-1}{d+2}}+\sqrt{\frac{2cK^2}{N}}\left(c_0N^{\frac{-1}{d+2}}\right)^{-d/2}
    = 2c_0N^{\frac{-1}{d+2}} +c_1 N^{\frac{-1}{d+2}}
\end{equation*}
where $c_1$ is the product of some constants that don't depend on $N$. This shows that
\begin{equation*}
    \E [W_1(\mu^N,\bar V)]\leq (2c_0 +c_1) N^{\frac{-1}{d+2}} \rightarrow 0
\end{equation*}
as $n\rightarrow\infty$.\\
\\
\textbf{Convergence a.s.}\\
It remains to show that $W_1(\mu^N,\bar V)$ goes to zero almost surely. For a function $f:M^N\rightarrow \R$ define
\begin{equation*}
    D_if(p_1,..,p_N)=\sup_{z\in M} f(p_1,..,p_{i-1},z,p_{i+1},..,p_N) - \inf_{z\in M} f(p_1,..,p_{i-1},z,p_{i+1},..,p_N).
\end{equation*}
Further, define the function $H:M^N\rightarrow\R$ by
\begin{equation*}
    (p_1,..,p_N)\mapsto \sup_{g\in \mathscr{F}_1} \left\{\frac{1}{N}\sum_{i=1}^N g(p_i) - \int_M g\dd\bar V\right\}.
\end{equation*}
Note that $H(p_1,..,p_N)=W_1(\mu^N,\bar V)$. 
\begin{lemma}\label{subgaussian}
Set (as before) $K=\sup_{p,q\in M}d(p,q)$. Then for each $1\leq j\leq N$: $||D_jH||_\infty\leq K/N$.
\end{lemma}
\begin{proof}
Let $1\leq j\leq N$ and fix $p_1,..,p_N$. Denote for $p\in M$ and $g\in \mathscr{F}_1$
\begin{equation*}
    J^j(g,p)=\frac{1}{N}\left(\sum_{i=1,i\neq j}^N g(p_i)+g(p)\right)-\int_Mg\dd\bar V
\end{equation*}Now let $p,q\in M$. Then for any $g\in \mathscr{F}_1$:
\begin{equation*}
    |J^j(g,p)-J^j(g,q)| = \frac{1}{N} |g(p)-g(q)| \leq \frac{1}{N} d(p,q) \leq \frac{K}{N}.
\end{equation*}
This shows that $g\mapsto J^j(g,p)$ and $g\mapsto J^j(g,q)$ are always at most $K/N$ apart from each other, which implies that
\begin{equation*}
    \left|\sup_{g\in \mathscr{F}_1}J^j(g,p)-\sup_{g\in \mathscr{F}_1}J^j(g,q)\right|\leq \frac{K}{N}.
\end{equation*}
Now 
\begin{eqnarray*}
    D_iH(p_1,..,p_N) &=& \sup_{p\in M} H(p_1,..,p_{i-1},p,p_{i+1},..,p_N) - \inf_{q\in M} H(p_1,..,p_{i-1},q,p_{i+1},..,p_N)\\
    &=& \sup_{p,q\in M} | H(p_1,..,p_{i-1},p,p_{i+1},..,p_N) - H(p_1,..,p_{i-1},q,p_{i+1},..,p_N)|\\
    &=& \sup_{p,q\in M} \left|\sup_{g\in \mathscr{F}_1}J^j(g,p)-\sup_{g\in \mathscr{F}_1}J^j(g,q)\right|\leq \frac{K}{N}.
\end{eqnarray*}
Since $P_1,..,P_N$ were arbitrary, we conclude that $||D_jH||_\infty\leq \frac{K}{N}$.
\end{proof}
Now we are in position to prove the main result.
\begin{prop}
$W_1(\mu^N,\bar V)\rightarrow 0$ almost surely as $N\rightarrow\infty$.
\end{prop}
\begin{proof}
Since $P_1,..,P_N$ are independent, \cite[Theorem 3.11]{handel2016prob} gives us that for any $t>0$
\begin{eqnarray*}
    \p(W_1(\mu^N,\bar V)-\E W_1(\mu^N,\bar V)>t) &=& \p\left(H(P_1,..,P_N)-\E H(P_1,..,P_N)>t\right)\\
    &\leq& \exp\left( \frac{-2t^2}{\sum_{k=1}^N||D_kH||_\infty^2}\right) \leq \exp\left(\frac{-2t^2N}{K^2}\right),
\end{eqnarray*}
where the last inequality follows from lemma~\ref{subgaussian}. For reasons of symmetry we obtain
\begin{equation*}
    \p(\left|W_1(\mu^N,\bar V)-\E W_1(\mu^N,\bar V)\right|>t) \leq 2 \exp\left(\frac{-2t^2N}{K^2}\right).
\end{equation*}
By a standard application of the Borel-Cantelli lemma, this implies that $W_1(\mu^N,\bar V)-\E W_1(\mu^N,\bar V)\rightarrow 0$ a.s. Since we have already seen that $\E W_1(\mu^N,\bar V)\rightarrow 0$, we conclude that a.s. as $N\rightarrow\infty$
\begin{equation*}
    W_1(\mu^N,\bar V)\rightarrow 0.
\end{equation*}
\end{proof}

We conclude that sampling uniformly from the manifold yields a suitable grid with probability $1$. %\newpage
\section{Hydrodynamic limit of the SEP}\label{part1}

In section~\ref{part2} we showed the existence of uniformly approximating grids. In this section we will apply such grids. We will use it to define an interacting particle system on the manifold. Then we will show that this interacting particle system has a hydrodynamic limit and that this limit satisfies the heat equation (the precise formulation is given in theorem~\ref{mainthm}). We follow a standard method that is used in \cite[Chapter 8]{seppalainen2008translation} for the Euclidean case.\\
Now let $(G_N,W_N)_{N=1}^\infty$ be a sequence of uniformly approximating grids with corresponding weights. Recall that this means the following. There is a sequence $(p_n)_{n=1}^\infty$ in $M$ such that $G^N=\{p_1,..,p_N\}$. On each $G^N$, there is a random walk $X^N$ which jumps from $p_i$ to $p_j$ with (symmetric) rate $W^N_{ij}$. We assume that there exists some function $a:\N\rightarrow[0,\infty)$ and some constant $C>0$ such that for each smooth $\phi$
\begin{equation*}
    a(N)\sum_{j=1}^NW^N_{ij}(\phi(p_j)-\phi(p_i))\longrightarrow C\Delta_M\phi(p_i)\quad(N\rightarrow\infty)
\end{equation*}
where the convergence is in the sense that for all smooth $\phi$
\begin{equation}\label{suitableGridAssumption}
    \lim_{N\rightarrow\infty} \frac{1}{N}\sum_{i=1}^N \left|a(N)\sum_{j=1}^NW^N_{ij}(\phi(p_j)-\phi(p_i))- C\Delta_M\phi(p_i)\right|=  0.\footnote{Recall from remark~\ref{sgconv} that if we replace the average in this expression by a supremum, this condition implies convergence of the corresponding semigroups.}
\end{equation}
By dividing $a(N)$ by $C$ if necessary, we can assume that $C=1$.

\begin{rmk}
Note that for the result of this section it is not necessary to construct grids from a sequence. Any sequence of finite grids such that~(\ref{suitableGridAssumption}) holds would do. However, since the grid that we constructed in section~\ref{part2} is of this form and this section partially serves as an example of the application of that grid, we formulate our results in this section in the same way.
\end{rmk}

\subsection{Symmetric Exclusion Process}\label{SEP}
The Symmetric Exclusion Process (SEP) is an interacting particle system that was introduced in~\cite{spitzer1970interaction} and studied in detail in~\cite[Chapter 8]{liggett2012interacting}. The idea is that there is some (possibly countably infinite) amount of particles on a (possibly countably infinite) graph $G$. The particles are considered identical. Each particle jumps after independent exponential times with parameter $1$ from $x$ to $y$ with probability $p(x,y)$, provided that the place that it wants to jump to is not already occupied. Otherwise, the jump is suppressed. We assume that $p(x,y)=p(y,x)$. Let $\eta_t\in \{0,1\}^G$ denote the configuration of the particles at time $t$, i.e. $\eta_t(x)=1$ if there is a particle at place $x\in G$ at time $t$ and $0$ else. We will sometimes write $\eta(p,t)=\eta_t(p)$. For any configuration $\eta$ and points $x,y$ define $\eta^{xy}$ by
\begin{equation*}
    \eta^{xy}(z)=\begin{cases} \eta(x) & \text{ if } z=y\\
    \eta(y) & \text{ if } z=x\\
    \eta(z) & \text{ if } z\neq x,y
    \end{cases}
\end{equation*}
An equivalent description of this process is the following. All edges $(xy)$ have independent exponential clocks with rate $p(x,y)=p(y,x)$. Whenever a clock rings, the particles that are at either side of the corresponding edge jump along the edge. This means that if there are no particles, nothing happens. If there is one particle, it jumps. If there are two particle, they switch places. Since we are not interested in individual particles, the configuration stays the same in the latter case. Note that in this way there can never be more than two particles at the same place. Using the notation introduced above, we see that the generator of this process is defined on the core of local functions as
\begin{equation*}
    Lf(\eta)=\frac{1}{2}\sum_{x,y} p(x,y) (f(\eta^{xy})-f(\eta)).
\end{equation*}
The factor $\frac{1}{2}$ is there since we count every edge twice.\\
\\
\textbf{The process}\\
We now define the SEP $\eta^{N}=(\eta^{N}_t)_{t\geq0}$ on $G^N$ through the generator
\begin{equation*}
    L^{N}h(\eta)=\frac{a(N)}{2}\sum_{i,j=1}^NW_{ij}^N (h(\eta^{ij})-h(\eta)), \quad h: \{0,1\}^{G^N}\rightarrow\R.
\end{equation*}
Here $\eta^{ij}:=\eta^{p_ip_j}$. It follows from our considerations above that this process describes particles that perform independent random walks according to $X^N$ with the restriction that jumps to occupied sites are suppressed.\\
Let $(X_i)_{i=1}^\infty$ be some sequence of (possibly degenerate) random variables taking values in $\{0,1\}$. Set as the initial configuration $\eta^{N}_0(p_i)=X_i$.\\
\\
\subsection{Hydrodynamic limit}\label{hydlim}
We will use this subsection to give the basic definitions that describe the idea of a hydrodynamic limit. At a microscopic scale, the particles are just random walkers with some interaction, but at the macroscopic scale (where limits are taken in space and time), the behaviour is deterministic: it is described by a partial differential equation (in our case the heat equation). \\
\\
%\textbf{Radon measures}\\
%We start by introducing Radon measures, since we will use these to describe particle %configurations.
%\begin{defi}
%A \textit{Radon measure} is a Borel measure $\mu$ on a Hausdorff topological space that %satisfies the following properties. 
%\begin{itemize}
%    \item $\mu$ is locally finite, meaning that every point has a neighbourhood with finite measure.
%    \item $\mu$ is inner regular, i.e. for any measurable set $V$: $\mu(V)=\sup\{\mu(K): K\subset V \text{ compact}\}$.
%\end{itemize}
%\end{defi} 
%Since $M$ is a metric space, it is Hausdorff, we can define Radon measures on it. Denote the space of Radon measures on $M$ by $R(M)$. Denote $\mu(f)=\int_M f\dd\mu$ for any $f:M\rightarrow \R$ and $\mu\in R(M)$ for which $\int_M f\dd\mu$ is defined. The vague (or weak) topology on $R(M)$ is defined as follows. 
%\begin{defi}
%We say that a sequence of measures $(\mu_N)_{N=1}^\infty$ in $R(M)$ converges \textit{vaguely} or \textit{weakly} to $\mu\in R(M)$ as $N$ goes to infinity if for any continuous $\phi$:
%\begin{equation*}
%    \mu_N(\phi)\rightarrow \mu(\phi) \text{ as } N\rightarrow\infty.
%\end{equation*}
%\end{defi}
%In general, we would need to restrict to continuous and bounded $\phi$, but since $M$ is compact, any continuous function is bounded. Since $M$ is a Polish space (i.e. complete and separable), it can be shown that $R(M)$ with the vague topology is a Polish space too.\\
%\\
\textbf{Path space}\\
Now write $R(M)$ for the space of Radon measures on $M$ with the vague topology and let $D=D([0,\infty),R(M))$ denote the space of all paths $\gamma:[0,\infty)\rightarrow R(M)$ such that $\gamma$ is right continuous and has left limits. On this space we can define the Skohorod metric (see for instance~\cite[Appendix A.2.2]{seppalainen2008translation}). Since $R(M)$ is a Polish space, it can be shown that $D$ with the Skohorod metric is a Polish space too. \\
%In particular, by Prokhorov's theorem, this implies that a subset of $D$ is tight if and only if its closure is compact.
\\
\textbf{Initial conditions and trajectories of particle configurations}\\
Define
\begin{equation*}
    \mu_t^{N}=\frac{1}{N}\sum_{i=1}^N\delta_{p_i}\eta^{N}_{t}(p_i),
\end{equation*}
where $\delta_p$ is the Dirac measure which places mass 1 at $p\in M$. It puts a point mass at each particle and rescales it by the amount of possible positions, which represents the particle configuration $\eta^{N}_t$ at time $t$. In particular $\mu_t^{N}$ is a sub-probability measure and is in $R(M)$. \\
Instead of dealing with this problem pointwise for each $t$, we will look at trajectories. As the particles move according to the SEP, $\gamma^{N}:[0,\infty)\rightarrow R(M)$ defined by $t\mapsto \mu_t^{N}$ is a random trajectory and hence a random element of $D$. It represents the positions of the particles over time. The initial configuration $X_1,..,X_N$ and the dynamics of the SEP determine a distribution $Q^{N}$\label{QNdef} on $D$. In this way we obtain a sequence $(Q^{N})_{N=0}^\infty$ of measures on $D$.\\
\\
\textbf{Assumption on the initial configuration}\\
We assume that there exists a measurable function $\rho_0:M\rightarrow \R$ such that $0\leq \rho_0\leq 1$ and $\mu_0^{N}$ converges vaguely to $\rho_0\dd \bar V$ in probability, i.e. for any continuous $\phi$ as $N\rightarrow\infty$:
\begin{equation}\label{incon}
    \int_M \phi \dd\mu_0^{N} \rightarrow \int_M \rho_0\phi \dd \bar V \hspace{0.3cm} \text{ in probability.}
\end{equation}
If this is the case, we say that $\rho_0\dd V$ is the density profile corresponding to the configurations $\eta_0^{N}$. Note that using measures here to represent the particles provides a bridge between separate particles (discrete measures) and density profiles (measures that are absolutely continuous with respect to $V$). We would like to show that if this initial condition is given, then at any time $t$ the configurations $\eta_t^{N}$ have a corresponding density profile $\rho_t\dd \bar V$. Moreover, we want to show that $t\mapsto\rho_t$ solves the heat equation with initial condition $\rho_0$. \\
\\
\textbf{Example of initial distribution}\\
Suppose for now that the $p_i$'s are such that for any continuous $f$: $\frac{1}{N}\sum_{i=1}^Nf(p_i)\rightarrow \int_Mf\dd \bar V$
\footnote{Since Kantorovich convergence is stronger than convergence in distribution, this is in particular true for the grids that we consider in part 2.}. Define the random variables $(X_i)_{i=1}^\infty$ to be independent Bernoulli random variables with $\E X_i = \rho_0(p_i)$ for some continuous function $\rho_0:M\rightarrow \R$ with $0\leq \rho_0\leq 1$. Then we see as $N\rightarrow\infty$:
\begin{eqnarray*}
    \E\left[\int \phi \dd\mu_0^{N}\right] &=& \E \left[\frac{1}{N} \sum_{i=1}^N \phi(p_i)\eta_0^{N}(p_i)\right] = \frac{1}{N} \sum_{i=1}^N \phi(p_i)\E \eta_0^{N}(p_i)\\
    &=& \frac{1}{N} \sum_{i=1}^N \phi(p_i)\rho_0(p_i)\rightarrow \int \phi\rho_0\dd \bar V, 
\end{eqnarray*}
since $\phi$ and $\rho_0$ are continuous. Further,
\begin{eqnarray*}
    \var \left[\int \phi \dd\mu_0^{N}\right] &=& \var \left[\frac{1}{N} \sum_{i=1}^N \phi(p_i)\eta_0^{N}(p_i)\right] = \frac{1}{N^2} \sum_{i=1}^N \phi(p_i)\var (\eta_0^{N}(p_i)) \\
    &=& \frac{1}{N^2} \sum_{i=1}^N \phi(p_i)\rho_0(p_i)(1-\rho_0(p_i))\rightarrow 0.
\end{eqnarray*}
Together this implies that~(\ref{incon}) holds here for any continuous $\phi$. \\
\\
\textbf{Main result}\\
After all these definitions, we can state the main result of this section.
\begin{thm}\label{mainthm}
Let $M$ be a complete, $n$-dimensional, connected Riemannian manifold and let $(G_N,W_N)_{N=1}^\infty$ be a sequence of uniformly approximating grids with corresponding weights. Let $\eta^N_t$ be particle configurations that behave according to the SEP on $(G_N,W_N)$ and let $\mu^N_t$ be its measure valued representation. Suppose that $\mu_0^N$ has density profile $\rho_0\dd V$ for some measurable function $\rho_0$. Then the trajectory $t\mapsto \mu^N_t$ converges in probability to the trajectory $t\mapsto \rho_t\dd V$ in the Skohorod topology, where $t\mapsto \rho_t$ satisfies the heat equation on $M$ with initial condition $\rho_0$.
\end{thm}

\subsection{Convergence result}\label{proof}
\textbf{Dynkin martingale}\\
The proof of the hydrodynamic result follows the line of~\cite[Chapter 8]{seppalainen2008translation}. Its core calculations are based on the following Dynkin martingale result. It is a standard result and it is also proved in~\cite{seppalainen2008translation}. We will formulate it in terms of our situation on a compact Riemannian manifold.
\begin{prop}\label{mgalethm}
Let $\{\eta_t,t\geq 0\}$ be a Feller process on a compact Riemannian manifold with generator $L$ and semigroup $S_t$. For any function $f$ such that both $f$ and $f^2$ are in $D(L)$, define
\begin{equation*}
    M_t=f(\eta_t)-f(\eta_0) - \int_0^t Lf(\eta_s)\dd s.
\end{equation*}
Then $M_t$ is a martingale with respect to the filtration $\mathscr{F}_t=\sigma\{\eta_r,r\leq t\}$. Moreover, its quadratic variation $\left<M,M\right>_t$ equals $\int_0^t\gamma(s)\dd s$, where $\gamma(s)=(L(f^2)-2fLf)(\eta_s)$.
\end{prop}\sk

\textbf{Application of the proposition}\\
First of all fix a smooth function $\phi$ on $M$. Define for $\eta\in \{0,1\}^{G^N}$: $f^N(\eta)=\frac{1}{N}\sum_{i=1}^N \eta(p_i)\phi(p_i)=\mu(\phi)$, where $\mu=\frac{1}{N}\sum_{i=1}^n\delta_i\eta(p_i)$. Note that since $L^{N}$ is the generator of a random walk on a the finite space of configurations, its domain consists of all functions on those configurations, so in particular $f^N$ and $(f^N)^2$ are in it. Applying theorem~\ref{mgalethm} in this situation shows that $M^{N}$ defined by
\begin{equation}\label{mgaleeq1}
M^{N}_{t}=f^N(\eta^{N}_{t})-f^N(\eta^{N}_0) - \int_0^{t} L^{N}f(\eta^{N}_s)\dd s
\end{equation}
is a martingale with quadratic variation $\left<M^{N},M^{N}\right>_t=\int_0^t\gamma(s)\dd s$, where $\gamma(s)=(L^{N}(f^N)^2-2f^NL^{N}f^N)(\eta_s)$. Some basic manipulations show that 
\begin{equation}
    f^N(\eta^{ij})-f^N(\eta)=-\frac{1}{N}(\phi(p_j)-\phi(p_i))(\eta(p_j)-\eta(p_i).\label{feta}
\end{equation}
Inserting definitions and leaving out some indexes (to keep everything clear) shows that the right hand side of~(\ref{mgaleeq1}) equals
\begin{eqnarray}
    &&\frac{1}{N}\sum_{i=1}^N\phi(p_i)(\eta_t(p_i))-\frac{1}{N}\sum_{i=1}^N\phi(p_i)(\eta_0(p_i))\nonumber\\
    &&-\left(-\int_0^{t} \frac{a(N)}{2N} \sum_{i,j=1}^NW_{ij}^N (\phi(p_j)-\phi(p_i))(\eta_s(p_j)-\eta_s(p_i))\dd s\right)\nonumber\\
    &=&\mu_t^{N}(\phi)-\mu_0^{N}(\phi) - \int_0^{t} \frac{a(N)}{N} \sum_{i,j=1}^NW_{ij}^N(\phi(p_j)-\phi(p_i))\eta_{s}(p_i)\dd s\nonumber\\
    &=&\mu_t^{N}(\phi)-\mu_0^{N}(\phi) - \int_0^{t} \frac{1}{N} \sum_{i=1}^N\eta_{s}(p_i)\left(a(N)\sum_{j=1}^NW_{ij}^N(\phi(p_j)-\phi(p_i))\right)\dd s.\label{maineq}
\end{eqnarray}\sk

\textbf{Using convergence of the generators}\\
By~(\ref{suitableGridAssumption}), we can write for any $p_i$:
\begin{equation}\label{usegenconv}
    a(N)\sum_{j=1}^NW_{ij}^N(\phi(p_j)-\phi(p_i)) = \Delta_M\phi(p_i)+E_{p_i}(N),
\end{equation}
where 
\begin{equation}\label{resttozero}
   E(N):=\frac{1}{N}\sum_{i=1}^N|E_{p_i}(N)|\rightarrow 0 \hspace{1cm}(N\rightarrow\infty). 
\end{equation}
This shows that
\begin{eqnarray*}
    &&\int_0^{t} \frac{1}{N} \sum_{i=1}^N\eta_{s}(p_i)\left(a(N)\sum_{j=1}^NW_{ij}^N(\phi(p_j)-\phi(p_i))\right)\dd s \\
    &=& \int_0^{t} \frac{1}{N} \sum_{i=1}^N\eta_{s}(p_i)\left(\Delta_M\phi(p_i)+E_{p_i}(N)\right)\dd s\\
    &=& \int_0^{t} \frac{1}{N} \sum_{i=1}^N\eta_{s}(p_i)\Delta_M\phi(p_i)\dd s +\int_0^{t} \frac{1}{N} \sum_{i=1}^N\eta_{s}(p_i)E_{p_i}(N)\dd s\\
    &=& \int_0^{t} \mu_s(\Delta_M\phi)\dd s +\int_0^{t} \frac{1}{N} \sum_{i=1}^N\eta_{s}(p_i)E_{p_i}(N)\dd s.
\end{eqnarray*}
Plugging this into~(\ref{maineq}) and~(\ref{mgaleeq1}), we obtain:
\begin{equation}\label{measure_equation}
    \mu_t^{N}(\phi)-\mu_0^{N}(\phi) - \int_0^{t} \mu_s^{N}(\Delta_M\phi)\dd s = M_t^{N} + \int_0^{t} \frac{1}{N} \sum_{i=1}^N\eta_{s}^{N}(p_i)E_{p_i}(N)\dd s,
\end{equation}
so for any $T>0$:
\begin{equation}
    \sup_{0\leq t\leq T}\left|\mu_t^{N}(\phi)-\mu_0^{N}(\phi) - \int_0^{t} \mu_s^{N}(\Delta_M\phi)\dd s \right| \leq \sup_{0\leq t\leq T} \left|M_t^{N}\right| + \sup_{0\leq t\leq T}\left|\int_0^{t} \frac{1}{N} \sum_{i=1}^N\eta_{s}^{N}(p_i)E_{p_i}(N)\dd s\right|.\label{supeq}
\end{equation}
We want to show that this expression converges to $0$ in probability. We will deal with the terms on the right hand side separately. \\
\\
\textbf{The error term}\\
First of all
\begin{eqnarray*}
    \left|\int_0^{t} \frac{1}{N} \sum_{i=1}^N\eta_{s}^{N}(p_i)E_{p_i}(N)\dd s\right|
    &\leq&\int_0^{t} \frac{1}{N} \sum_{i=1}^N|\eta_{s}^{N}(p_i)| |E_{p_i}(N)|\dd s
    \leq \int_0^{t} E(N)\dd s\\
    &=& t E(N),
\end{eqnarray*}
so
\begin{equation*}
     \sup_{0\leq t\leq T}\left|\int_0^{t} \frac{1}{N} \sum_{i=1}^N\eta_{s}(p_i)E_{p_i}(N)\dd s\right|
     \leq TE(N)\rightarrow 0 \quad \text{(by (\ref{resttozero}))}.
\end{equation*}\sk

\textbf{Convergence of the martingale to $0$}\\
Now for the other term. Since the trajectory $t\mapsto \mu_t^{N}$ is cadlag, so is $M^{N}$. Hence by Doob's inequality we see:
\begin{equation}
    \p\left(\sup_{0\leq t\leq T} \left|M_t^{N}\right|>\delta\right)\leq \frac{\E|M_T^{N}|}{\delta}.\label{Doob}
\end{equation}
To show that $\E|M_T^{N}|$ goes to $0$, it suffices to show that $\E\left<M^{N},M^{N}\right>_T$ goes to $0$ (since then $\E\left[(M_T^{N})^2\right]=\E\left<M^{N},M^{N}\right>_T\rightarrow 0$ and hence $\E|M_T^{N}|\rightarrow 0$). This is what the following lemma tells us.
\begin{lemma}\label{quadvar}
For any $T>0$:
\begin{equation*}
    \lim_{N\rightarrow\infty}\E\left<M^{N},M^{N}\right>_T=0.
\end{equation*}
\end{lemma}
\begin{proof}
Recall that $\left<M^{N},M^{N}\right>_T=\int_0^T (L^{N}(f^N)^2-2f^NL^{N}f^N)(\eta_s) \dd s$.
By writing out, one simply obtains
\begin{equation*}
    (L^{N}(f^N)^2-2f^NL^{N}f^N)(\eta)=\sum_{i,j=1}^N \frac{a(N)}{2} W_{ij}^N (f(\eta^{ij})-f(\eta))^2.
\end{equation*}
Using~(\ref{feta}), we see
\begin{equation*}
    (f(\eta^{ij})-f(\eta))^2\leq \left(\frac{1}{N}(\phi(p_j)-\phi(p_i))(\eta(p_j)-\eta(p_i))\right)^2
    \leq \frac{1}{N^2}(\phi(p_j)-\phi(p_i))^2,
\end{equation*}
since $\eta(p_i)\in\{0,1\}$ for all $i$. This shows that
\begin{eqnarray*}
    0&\leq& \left<M^{N},M^{N}\right>_T=\int_0^T (L^{N}(f^N)^2-2f^NL^{N}f^N)(\eta_s) \dd s \\
    &\leq& \int_0^T \frac{a(N)}{2N^2}  \sum_{i,j=1}^N W_{ij}^N (\phi(p_j)-\phi(p_i))^2\dd s
    = T \frac{a(N)}{2N^2} \sum_{i,j=1}^N W_{ij}^N (\phi(p_j)-\phi(p_i))^2.    
\end{eqnarray*}
This implies that also
\begin{equation}\label{quadvarest}
    0\leq\E\left<M^{N},M^{N}\right>_T\leq T \frac{a(N)}{2N^2} \sum_{i,j=1}^N W_{ij}^N (\phi(p_j)-\phi(p_i))^2.
\end{equation}
We can estimate this term by using~(\ref{resttozero}). Some basic manipulations show that
\begin{eqnarray*}
    &&\frac{a(N)}{2}\sum_{i,j=1}^N W_{ij}^N (\phi(p_j)-\phi(p_i))^2 = - \sum_{i=1}^N \phi(p_i) a(N) \sum_{j=1}^N W_{ij}^N (\phi(p_j)-\phi(p_i)) \\
    &=& - \sum_{i=1}^N \phi(p_i) \left(\Delta_M\phi(p_i)+E_{p_i}(N)\right) = - \sum_{i=1}^N \phi(p_i)\Delta_M\phi(p_i) - \sum_{i=1}^N \phi(p_i)E_{p_i}(N),
\end{eqnarray*}
where the $E_{p_i}$'s are as before. This implies that 
\begin{eqnarray*}
    &&\limsup_{N\rightarrow\infty}\left|\frac{a(N)}{2N^2}\sum_{i,j=1}^N W_{ij}^N (\phi(p_j)-\phi(p_i))^2\right| \leq \limsup_{N\rightarrow\infty}\left\{ \frac{1}{N^2}\sum_{i=1}^N |\phi(p_i)||\Delta_M\phi(p_i)| +\frac{1}{N^2} \sum_{i=1}^N |\phi(p_i)||E_{p_i}(N)|\right\} \\
    &\leq &\limsup_{N\rightarrow\infty}\frac{1}{N} ||\phi||_\infty ||\Delta_M\phi||_\infty + \limsup_{N\rightarrow\infty}\frac{1}{N}||\phi||_\infty E(N)=0,
\end{eqnarray*}
where in the last step we used~(\ref{resttozero}). So we obtain
\begin{equation*}
    \lim_{N\rightarrow\infty}\frac{a(N)}{2N^2} \sum_{i,j=1}^N W_{ij}^N (\phi(p_j)-\phi(p_i))^2 = 0.
\end{equation*} 
Together with~(\ref{quadvarest}) this gives the result.
%\begin{equation*}
%    0\leq \lim_{N\rightarrow\infty} \E\left<M^{N},M^{N}\right>_T \leq
%    \lim_{N\rightarrow\infty} \frac{T}{N}\frac{a(N)}{N} \sum_{i,j=1}^N W_{ij}^N (\phi(p_j)-\phi(p_i))^2 = 0.
%\end{equation*} 
\end{proof}
We conclude from the lemma that the right hand side of~(\ref{Doob}) goes to zero as $N$ goes to infinity and $\epsilon$ goes to zero, so 
\begin{equation*}
    \lim_{\epsilon\downarrow 0}\lim_{N\rightarrow\infty} \sup_{0\leq t\leq T} \left|M_t^{N}\right| = 0 \hspace{0.3cm} \text{in probability.}
\end{equation*}\sk

\textbf{Convergence of~(\ref{supeq}) to $0$ in probability}\\
Combining everything above and using~(\ref{supeq}), we conclude that 
\begin{equation*}
    \lim_{N\rightarrow\infty}\sup_{0\leq t\leq T}\left|\mu_t^{N}(\phi)-\mu_0^{N}(\phi) - \int_0^{t} \mu_s^{N}(\Delta_M\phi)\dd s \right|=0 \hspace{0.3cm} \text{in probability.}
\end{equation*}
In particular, for any $\delta\geq0$, define
\begin{equation*}
    H^\delta=\left\{ \alpha\in D: \sup_{0\leq t< T}\left|\alpha_t(\phi)-\alpha_0(\phi) - \int_0^{t} \alpha_s(\Delta_M\phi)\dd s \right| \leq \delta\right\}.
\end{equation*}
It can be shown, as in~\cite[Chapter 8]{seppalainen2008translation}, that $H^\delta$ is closed for any $\delta>0$.
Recall from page~\pageref{QNdef} that we write the distribution of $t\mapsto\mu_t^{N}$ as $Q^{N}$.
Then the convergence result above implies that for any $\delta>0$:
\begin{equation*}
    \lim_{N\rightarrow\infty}Q^{N}(H^\delta) = 1.
\end{equation*}\sk

\textbf{Tightness of $(Q^N)_{N=1}^\infty$}\\
We will need that the sequence of distributions $(Q^N)_{N=1}^\infty$ is tight. This can be shown in exactly the same way as~\cite[p.55-56]{kipnis1999scaling}. In fact all the most crucial calculations have already been performed above. 
\begin{lemma}
The sequence of distributions $(Q^N)_{N=1}^\infty$ is tight.
\end{lemma}
\begin{proof}
It needs to be shown that the two conditions of~\cite[Chapter 4 Thm 1.3]{kipnis1999scaling} are satisfied. Note that for any continuous $f$ we can map a path $\nu\in D([0,T],R(M))$ to the path in $D([0,T],\R)$ given by $t\mapsto \nu_t(f)$. This induces a sequence of distributions $Q^Nf^{-1}$ on $D([0,T],\R)$. By~\cite[Chapter 4 Prop 1.7]{kipnis1999scaling} and the fact that the smooth functions are uniformly dense in the set of continuous functions on a manifold, it suffices to prove the conditions of~\cite[Chapter 4 Thm 1.3]{kipnis1999scaling} for \{$Q^Nf^{-1},N\geq 0\}$ for all smooth $f$. Fix such $f$. Since each path stays in the set of sub-probability measures, the first condition is easily satisfied. For the second condition, it suffices to prove Aldous' tightness criterion, i.e. that
\begin{equation}\label{aldous}
    \lim_{\gamma\rightarrow 0} \limsup_{N\rightarrow \infty}\sup_{\tau\in \mathcal{I}_T, \theta\leq \gamma} Q^Nf^{-1} \left[ \left|\mu^N_\tau(f)-\mu^N_{\tau+\theta}(f)\right|>\epsilon\right]=0,
\end{equation}
where $\mathcal{I}_T$ denotes the set of all stopping times bounded by $T$. We know from equation~(\ref{measure_equation}) that there exists a martingale $M$ (depending on $f$) such that
\begin{equation*}
    \mu_t^{N}(f)-\mu_0^{N}(f) - \underbrace{\int_0^{t} \mu_s^{N}(\Delta_M f)\dd s}_\text{(I)} = \underbrace{M_t^{N}}_\text{(II)} + \underbrace{\int_0^{t} \frac{1}{N} \sum_{i=1}^N\eta_{s}^{N}(p_i)E_{p_i}(N)\dd s}_\text{(III)}.
\end{equation*}
It therefore suffices to check the tightness criterion for the RHS of this equation and for the integral on the LHS (since the only other term is constant). Now we can make the following estimations. \\
\\
\textbf{(I)}. First of all, since $\mu_s^N$ is a sub-probability measure and $\Delta_M f$ is bounded:
\begin{equation*}
    \left|\int_0^{\tau+\theta} \mu_s^{N}(\Delta_M f)\dd s - \int_0^{\tau} \mu_s^{N}(\Delta_M f)\dd s\right|\leq \theta ||\Delta_M f||_\infty.
\end{equation*}
This implies that
\begin{eqnarray*}
    &&\sup_{\tau\in \mathcal{I}_T, \theta\leq \gamma} Q^Nf^{-1} \left[ \left|\int_0^{\tau+\theta} \mu_s^{N}(\Delta_M f)\dd s - \int_0^{\tau} \mu_s^{N}(\Delta_M f)\dd s\right|>\epsilon\right] \\
    &\leq&  Q^Nf^{-1} \left[ \sup_{\tau\in \mathcal{I}_T, \theta\leq \gamma} \left|\int_0^{\tau+\theta} \mu_s^{N}(\Delta_M f)\dd s - \int_0^{\tau} \mu_s^{N}(\Delta_M f)\dd s\right|>\epsilon\right] \\
    &\leq& Q^Nf^{-1} \left[ \sup_{\tau\in \mathcal{I}_T, \theta\leq \gamma} \theta ||\Delta_M f||_\infty >\epsilon\right] \leq  Q^Nf^{-1} \left[ \gamma ||\Delta_M f||_\infty >\epsilon\right] = \1_{\gamma ||\Delta_M f||_\infty >\epsilon}
\end{eqnarray*}
This implies that the limit in~(\ref{aldous}) is smaller than
\begin{equation*}
    \lim_{\gamma\rightarrow 0} \limsup_{N\rightarrow \infty} \1_{\gamma ||\Delta_M f||_\infty >\epsilon} = \lim_{\gamma\rightarrow 0} \1_{\gamma ||\Delta_M f||_\infty >\epsilon} = 0,
\end{equation*}
so (I) satisfies the tightness criterion.\\
\\
\textbf{(II)}. Further, the calculations above show that
\begin{equation*}
    \left|\int_0^{\tau+\theta} \frac{1}{N} \sum_{i=1}^N\eta_{s}^{N}(p_i)E_{p_i}(N)\dd s - \int_0^{\tau} \frac{1}{N} \sum_{i=1}^N\eta_{s}^{N}(p_i)E_{p_i}(N)\dd s\right|\leq \theta E(N)\leq \theta K.
\end{equation*}
Here $K$ is some positive number which exists, because of~(\ref{resttozero}). This part satisfies~(\ref{aldous}) in the same way as the previous part.\\
\\
\textbf{(III)}. Now for the last term, we first estimate $\E\left[(M^N_{\tau+\theta}-M^N_{\tau})^2\right]$ (as is done in~\cite[p.56]{kipnis1999scaling}). Naturally, the expectation is taken with respect to $Q^Nf^{-1}$. Note that because of the martingale property:
\begin{equation*}
    0\leq\E\left[(M^N_{\tau+\theta}-M^N_{\tau})^2\right] = \E(M^N_{\tau+\theta})^2-\E(M^N_{\tau})^2 = \E\left<M^N,M^N\right>_{\tau+\theta}-\E\left<M^N,M^N\right>_{\tau}.
\end{equation*}
We see from the calculations in the proof of lemma~\ref{quadvar} that
\begin{equation*}
    \E\left<M^N,M^N\right>_{\tau+\theta}-\E\left<M^N,M^N\right>_{\tau} \leq \theta \frac{a(N)}{2N^2} \sum_{i,j=1}^N W_{ij}^N (\phi(p_j)-\phi(p_i))^2.
\end{equation*}
Since the term after $\theta$ converges to $0$, we see that it is bounded by some constant $\alpha$. By Chebyshev's inequality we obtain:
\begin{equation*}
    Q^Nf^{-1} \left(|M^N_{\tau+\theta}-M^N_{\tau}|>\epsilon\right)\leq \frac{\E\left[(M^N_{\tau+\theta}-M^N_{\tau})^2\right]}{\epsilon^2} \leq \frac{\theta \alpha}{\epsilon^2}.
\end{equation*}
Since
\begin{equation*}
    \lim_{\gamma\rightarrow 0} \limsup_{N\rightarrow \infty}\sup_{\tau\in \mathcal{I}_T, \theta\leq \gamma} \frac{\theta \alpha}{\epsilon^2} = \lim_{\gamma\rightarrow 0} \limsup_{N\rightarrow \infty} \frac{\gamma \alpha}{\epsilon^2} =  \lim_{\gamma\rightarrow 0} \frac{\gamma \alpha}{\epsilon^2} = 0,
\end{equation*}
this part satisfies~(\ref{aldous}) too.
\end{proof}\sk

\textbf{Limit distribution}\\
We have just shown that $(Q^N)_{N=1}^\infty$ is a tight sequence of measures on $D$. This implies that every one of its subsequences is also tight and therefore has a weakly convergent subsequence. If these all have the same limit, then it follows from a basic result in metric spaces that the sequence itself converges weakly to that limit. It therefore suffices for weak convergence of $(Q^N)_{N=1}^\infty$ to show that every weakly convergent subsequence of $(Q^N)_{N=1}^\infty$ has the same limit. Let $(Q^{N_k})_{k=1}^\infty$ be any weakly convergent subsequence and denote its limit by $Q$. Since $H$ is closed, we know for any $\delta>0$ that
\begin{equation*}
    Q(H^\delta)\geq \limsup_{k\rightarrow\infty} Q^{N_k}(H^\delta)=1,
\end{equation*}
so $Q(H^\delta)=1$. Since this holds for any $\delta>0$, we see 
\begin{equation*}
    Q(H^0)=Q\left(\bigcap_{m=1}^\infty H^\frac{1}{m}\right)=1-Q\left(\bigcup_{m=1}^\infty (H^\frac{1}{m})^C\right)\geq 1- \sum_{m=1}^\infty Q\left(\left(H^\frac{1}{m}\right)^C\right)=1.
\end{equation*}
This means that
\begin{equation*}
    Q\left( \alpha\in D: \sup_{0\leq t< T}\left|\alpha_t(\phi)-\alpha_0(\phi) - \int_0^{t} \alpha_s(\Delta_M\phi)\dd s \right| =0 \right)=1.
\end{equation*}
By doing this for a countable set of functions $\phi$ that is dense in $C^\infty$ with respect to $||\cdot||_\infty+||\Delta_M\cdot||_\infty$ and arguing that this implies the same for any smooth function we see:
\begin{equation*}
    Q\left( \alpha\in D:  \sup_{0\leq t< T}\left|\alpha_t(\phi)-\alpha_0(\phi) - \int_0^{t} \alpha_s(\Delta_M\phi)\dd s \right| =0 \hspace{0.3cm}\forall\phi\in C^\infty \right)=1.
\end{equation*}
Since this holds for any $T>0$, we see that $Q-$a.s. for every $t\geq 0$ and for all smooth $\phi$:
\begin{equation}
    \alpha_t(\phi)-\alpha_0(\phi) = \int_0^{t} \alpha_s(\Delta_M\phi)\dd s.\label{whealpha}
\end{equation}
Note that~(\ref{whealpha}) is a weak, measure-valued formulation of the heat equation. We will argue and use shortly that this equation uniquely determines the trajectory $t\mapsto \alpha_t$ given the initial conditions.\\
\\
\textbf{Continuity}\\
To obtain uniqueness, we first need to know that the trajectory is continuous. For the $\R^n$ case this is shown in~\cite[Lemma 8.6]{seppalainen2008translation}. The result can be shown in exactly the same way in our case, so we will not provide all the details. The topology on the space of measures is generated by the following metric:
\begin{equation*}
    d_M(\mu,\nu) = \sum_{j=1}^\infty 2^{-j} \left( 1\wedge \left|\mu(\phi_j)-\nu(\phi_j)\right|\right),
\end{equation*}
for some sequence $\phi_j\in C^\infty(M)$. It suffices to control
\begin{equation*}
    \sup_{t\geq 0} \e^{-t}d_M(\mu^N_t,\mu^N_{t-}).
\end{equation*}
Doing that can be reduced to showing that for any $T>0$ and $\psi\in C^\infty(M)$:
\begin{equation*}
    \lim_{\delta\rightarrow 0}\limsup_{n\rightarrow\infty} \E\left[\sup_{0\leq s,t\leq T, |s-t|<\delta} \left|\mu^N_s(\phi)-\mu^N_t(\phi)\right|^2\right].
\end{equation*}
This can be done by using the Dynkin martingale representation~(\ref{measure_equation}) and bounding all the differences as in the proof of tightness. The only term that needs some attention is $(M^N_t-M^N_s)^2$, but it can be controlled using Doob's maximal inequality:
\begin{equation*}
    \E \left[\sup_{0\leq s,t\leq T, |s-t|<\delta} (M^N_t-M^N_s)^2\right] \leq \E \left[\sup_{0\leq t\leq T} 4 (M^N_t)^2\right] \leq 16 \E (M^N_T)^2 = 16  \E \left<M^N,M^N\right>_T,
\end{equation*}
which goes to zero according to lemma~\ref{quadvar}.\\

%which is smaller than
%\begin{equation*}
%    \sup_{0\leq s,t\leq T, |s-t|<\delta} d_M(\mu^N_t,\mu^N_{s}) + \exp^{-T}
%\end{equation*}
%for any $\delta,T>0$.\\
%By cutting of $d_M$ after $m$ terms, we can show that for any $m$ it is smaller than
%\begin{equation*}
%    \sup_{0\leq s,t\leq T, |s-t|<\delta} \sum_{j=1}^\infty \left|\mu^N_s(\phi_j)-\mu^N_t(\phi_j)\right| + 2^{-m} + \exp^{-T}.
%\end{equation*}
%Since only finitely many $\phi_j$ are left, it now suffices to show for any $T>0$ and $\phi\in C^\infty(M)$:
%\begin{equation*}
%    \lim_{n\rightarrow\infty}\lim_{\delta\rightarrow 0}
%\end{equation*}

\textbf{Uniqueness}\\
To obtain uniqueness of limits of subsequences of $Q^N$, we need to know that there is a unique continuous solution to~(\ref{whealpha}) that has initial condition $\rho_0\dd\bar V$. We know that $t\mapsto\rho_t\dd \bar V$ is a continuous solution to~(\ref{whealpha}) with the right initial condition if $t\mapsto \rho_t$ satisfies the heat equation with initial condition $\rho_0$. Therefore it suffices to show that this solution is unique. This result is proven with a boundedness condition in~\cite[Thm A.28]{seppalainen2008translation}. The main idea of the proof is that the measure valued path $\alpha_t$ is smoothed by taking its convolution with some smooth kernel with bandwidth $\epsilon>0$. Then it is shown that this trajectory of functions satisfies the heat equation with initial condition $\rho_0$ in the strong sense (by interchanging integral and derivatives and using that these identities are known for sufficiently many $\phi$), so it must equal $t\mapsto\rho_t$. Then by letting $\epsilon$ go to zero, it is shown that the original trajectory $t\mapsto\alpha_t$ must equal $t\mapsto\rho_t\dd \lambda$, where $\lambda$ is the Lebesgue measure. \\
\\
To obtain the analogous result in our setting, we cannot use convolution, since this is not well-defined on a manifold. However, we can smooth the measures by integrating the heat kernel at time $\epsilon$ with respect to the measures. Using this smoothing, we can follow exactly the same approach, i.e. showing that the smoothed trajectory satisfies the heat equation in a strong sense and then letting $\epsilon$ go to $0$. The boundedness condition is a bound on volumes, which is needed for some estimations in~\cite{seppalainen2008translation} and for the uniqueness of the strong solution to the heat equation. Since we work in a compact setting and with probability measures, such a bound is not necessary. The uniqueness of the strong solution to the heat equation is a standard result in our case (so for a compact and connected Riemannian manifold). See for instance \cite[Thm 8.18]{grigoryan2009heat}. Results on the heat kernel on a manifold can also be found in~\cite{grigoryan2009heat}.\\
% {\color{red} WOULD BE GOOD TO FIND SOME OTHER REFERENCE FOR MEASURE-VALUED INITIAL CONDITIONS}.\\
\\
\textbf{Conclusion}\\
Now let $t\mapsto\rho_t$ be the solution to the heat equation on $M$ with initial condition $\rho_0$ and call $\beta:=(t\mapsto\rho_t\dd \bar V)$. Recall that~(\ref{whealpha}) holds $Q-$a.s. By the uniqueness result above, this implies that $Q$ is a Dirac distribution with $\beta$ as its support. Since this does not depend on $Q^{N_k}$, it must be the same for any convergent subsequence, so with arguments given above, we conclude that $Q^N\rightarrow Q$ weakly. Let $\gamma^N$ denote the random trajectory $t\mapsto \mu^{N}_t$. Since $Q$ is degenerate, the weak convergence implies convergence in probability, so $\gamma^N\rightarrow \beta$ in probability. This is what we wanted to show. %\newpage
\section*{Acknowledgement}
The authors thank Rik Versendaal for helpful discussions. The support of the grant 613.009.112 of the Netherlands Organisation for Scientific Research (NWO) is gratefully acknowledged.

%-----------------------------------------------------------------
%now the real content is over
\phantomsection
\bibliographystyle{abbrvnat}
\bibliography{refs}
%\newpage

%\newpage
%\input{sections/DirichletForm/main.tex}
\section*{Appendix}

\begin{lemma*}\label{eventuallyconnected}
Let $(p_i)_{i=1}^\infty$ be a sequence for which the empirical measures converge to the volume measure in the Kantorovich sense. Define $\epsilon(N)$, $W^N_{ij}$ and $k$ as in section~\ref{1modmot}.
Additionally suppose that there exists some $\alpha>0$ such that $k(x)>0$ for all $x\leq \alpha$. Say that there is an edge between $p_i$ and $p_j$ whenever $W^N_{ij}>0$. Then the corresponding graphs are eventually connected (in other words: there is some $N_0$ such that for all $N\leq N_0$ $V_N$ with edges as just defined is connected).
\end{lemma*}
\begin{proof}
Define
\begin{eqnarray*}
    G_N(\beta) &:=& \text{the graph that is obtained from $V_N$ by putting an edge between vertices at distance }\leq \beta\\
    \beta_N&:=&\inf\{\beta\geq 0: G_N(\beta) \text{ is connected}\}.
\end{eqnarray*}
Since $G_N(0)$ is not connected (for $N>1$), $G_N(\sup_{p,q\in M}d(p,q))$ is connected and $G_N(\beta_1)$ contains all edges of $G_N(\beta_2)$ for $\beta_1\geq \beta_2$ it is clear that $\beta_N$ is a finite number strictly larger than $0$. Further note that $G_N(\beta_N)$ is connected (so the infimum is actually a minimum).

Now note that there must be two points $p',q'\in V_N$ such that $p,q$ have an edge between them for $\beta=\beta_N$ and are not connected for $\beta<\beta_N$ (we call $p$ and $q$ connected if there is a path from $p$ to $q$). Indeed if any pair $p,q\in V_N$ that has an edge between them for $\beta=\beta_N$ is still connected by some path for some $\beta_{pq}<\beta_N$, we see that for $\beta'=\sup_{p,q}\beta_{pq}<\beta_N$ the graph $G_N(\beta')$ is connected, which contradicts the definition of $\beta_N$ (note that the supremum ranges over a finite amount of numbers, since $V_N$ is finite). Fix such $p',q'\in V_N$.

Now let $s_N$ be a point on $M$ such that $d(p',s_N)=d(q',s_N)=\beta_N/2.$ Then $B(s_N,\beta_N/4)$ does not contain any point of $V_N$ (since by the triangle inequality such point would have distance $\leq 3\beta_N/4$ to both $p'$ and $q'$ so $p'$ and $q'$ would be connected to each other via this point in $G_N(3\beta_N/4)$, which contradicts the choice of $p'$ and $q'$).

Now we define the following function $l_N:M\rightarrow \R$
\begin{equation*}
    l_N(p) = \begin{cases}    d(p,s_N)-\frac{\beta_N}{4} & p\in B\left(s_N,\frac{\beta_N}{4}\right)\\ 0 & \text{otherwise} \end{cases}
\end{equation*}
It is easy to see that $|l_N(p)-l_N(q)|\leq d(p,q)$, so $l_N$ is Lipschitz with $L_{l_N}\leq 1$. This implies that
\begin{equation*}
    W_1(\mu^N,\overline V)\geq \int l_N \dd \mu_N - \int l_N \dd \overline V.
\end{equation*}
Since $l_N$ is only non-zero on $B(s_N,\beta_N/4)$ and this set does not contain points of $V_N$, we see that 
\begin{equation*}
    \int l_N \dd \mu_N=0.
\end{equation*}
Further, since $l_N$ is non-positive and $l_N\leq -\beta_N/8$ on $B(s_N,\beta_N/8$, we see that
\begin{equation*}
    \int l_N \dd \overline V \leq -\overline V\left(B\left(s_N,\frac{\beta_N}{8}\right)\right) \frac{\beta_N}{8},
\end{equation*}
so we conclude that $W_1(\mu^N,\overline V)\geq \overline V(B(s_N,\beta_N/8)) \beta_N/8$. Since $W_1(\mu^N,\overline V)$ goes to zero, it is easy to deduce from this inequality that $\beta_N\rightarrow 0$. Hence there are constants $C',C''>0$ (not depending on $s_N$), such that for $N$ large enough 
\begin{equation*}
    W_1(\mu^N,\overline V)\geq \overline V\left(B\left(s_N,\frac{\beta_N}{8}\right)\right) \frac{\beta_N}{8} \geq C''\left(\frac{\beta_N}{8}\right)^d \frac{\beta_N}{8} = C' \beta_N^{d+1}.
\end{equation*}
Now we see there is a $C>0$ such that for $N$ large enough
\begin{equation*}
    \epsilon_N=\left(\sup_{m\geq N} W_1(\mu^m,\bar V)\right)^{\frac{1}{4+d}} \geq W_1(\mu^N,\overline V)^{\frac{1}{4+d}} \geq C \beta_N^{\frac{d+1}{d+4}}.
\end{equation*}
This implies that there is some $N_0$ such that for all $N\geq N_0$ $\alpha\epsilon_N\geq\beta_N$. By our choice of $k$, all points at distance $\alpha\epsilon_N$ or less are joined by an edge, so this inequality combined with the definition of $\beta_N$ shows that for all $N\geq N_0$ $V_N$ with edges as defined in the lemma statement is connected.
\end{proof}
\end{document}